\documentclass[12pt,a4paper]{article}

\usepackage{lmodern}
\usepackage[top=2cm, bottom=2cm, left=2cm, right=2cm]{geometry}
\usepackage{tikz} 
\usetikzlibrary{matrix,arrows,decorations.pathmorphing,patterns,positioning,calc}
\usepackage{enumitem}

\usepackage[bookmarksnumbered = true,pdfstartview=FitH]{hyperref}
\hypersetup{
colorlinks=true,
breaklinks=true,
urlcolor= black,
linkcolor= black,
citecolor=black,
bookmarksopen=false}

\usepackage{amsmath}
\usepackage{amssymb}
\usepackage{amsthm}
\usepackage{mathtools}

\theoremstyle{plain}
\newtheorem{theorem}{Theorem}[section]
\newtheorem{proposition}[theorem]{Proposition}
\newtheorem{lemma}[theorem]{Lemma}

\newtheorem{remark}[theorem]{Remark}

\newtheorem{example}[theorem]{Example}

\newtheorem{theointro}{Theorem}

\newtheorem*{remintro}{Remark}
\newtheorem{propintro}[theointro]{Proposition}

\newcommand{\Z}{\mathbf{Z}}
\newcommand{\GL}{\mathrm{GL}}
\newcommand{\PGL}{\mathrm{PGL}}
\newcommand{\A}{\mathbf{A}}
\newcommand{\PP}{\mathbf{P}}
\newcommand{\G}{\mathbf{G}}
\newcommand{\Autgp}[1]{\mathbf{Aut}_{#1}}

\renewcommand{\phi}{\varphi}
\newcommand{\fonction}[4]{\left\{\begin{array}[c]{ccc}#1 & \longrightarrow & #2 \\ #3 & \longmapsto & #4 \end{array}\right.}
\newcommand{\Alb}[1]{\mathrm{Alb}(#1)}
\newcommand{\Albz}[1]{\mathrm{Alb}_0(#1)}
\newcommand{\Albu}[1]{\mathrm{Alb}_1(#1)}
\newcommand{\alb}[1]{a_{#1}}
\newcommand{\aff}[1]{#1_{\textrm{aff}}}
\newcommand{\ant}[1]{#1_{\textrm{ant}}}
\newcommand{\Ru}[1]{R_u(#1)}
\newcommand{\reg}[1]{#1_{\textrm{reg}}}

\DeclareMathOperator{\id}{id}
\DeclareMathOperator{\Spec}{Spec}
\DeclareMathOperator{\Aut}{Aut}
\DeclareMathOperator{\Hom}{Hom}
\DeclareMathOperator{\End}{End}
\DeclareMathOperator{\pr}{pr}
\DeclareMathOperator{\Gal}{Gal}
\DeclareMathOperator{\codim}{codim}
\DeclareMathOperator{\Ext}{Ext}
\DeclareMathOperator{\Pic}{Pic}
\DeclareMathOperator{\Sym}{Sym}
\DeclareMathOperator{\dom}{dom}

\DeclareMathOperator{\Lie}{Lie}
\DeclareMathOperator{\Vect}{Vect}
\DeclareMathOperator{\Proj}{Proj}

\newlist{mycase}{enumerate}{3}
\setlist[mycase]{align=left,listparindent=\parindent,parsep=\parskip,font=\normalfont\itshape\underline,leftmargin=0pt,labelwidth=0pt,itemindent=.4em,labelsep=.4em,partopsep=0pt,itemsep=\baselineskip}
\setlist[mycase,1]{label=Case~\arabic*.,ref=\arabic*}
\setlist[mycase,2]{label=Subcase~\themycasei.\arabic*.,ref=\themycasei.\arabic*}
\setlist[mycase,3]{label=Subsubcase~\themycaseii.\arabic*.,ref=\themycaseii.\arabic*}
\newlist{mystep}{enumerate}{3}
\setlist[mystep]{align=left,listparindent=\parindent,parsep=\parskip,font=\normalfont\itshape\underline,leftmargin=0pt,labelwidth=0pt,itemindent=.4em,labelsep=.4em,partopsep=0pt,itemsep=\baselineskip}
\setlist[mystep,1]{label=Step~\arabic*.,ref=\arabic*}
\setlist[mystep,2]{label=Case~\themystepi.\arabic*.,ref=\themystepi.\arabic*}
\setlist[mystep,3]{label=Subcase~\themystepii.\arabic*.,ref=\themystepii.\arabic*}

\makeatletter
\newcommand*{\da@rightarrow}{\mathchar"0\hexnumber@\symAMSa 4B }
\newcommand*{\da@leftarrow}{\mathchar"0\hexnumber@\symAMSa 4C }
\newcommand*{\xdashrightarrow}[2][]{%
  \mathrel{%
    \mathpalette{\da@xarrow{#1}{#2}{}\da@rightarrow{\,}{}}{}%
  }%
}
\newcommand{\xdashleftarrow}[2][]{%
  \mathrel{%
    \mathpalette{\da@xarrow{#1}{#2}\da@leftarrow{}{}{\,}}{}%
  }%
}
\newcommand*{\da@xarrow}[7]{%
  \sbox0{$\ifx#7\scriptstyle\scriptscriptstyle\else\scriptstyle\fi#5#1#6\m@th$}%
  \sbox2{$\ifx#7\scriptstyle\scriptscriptstyle\else\scriptstyle\fi#5#2#6\m@th$}%
  \sbox4{$#7\dabar@\m@th$}%
  \dimen@=\wd0 %
  \ifdim\wd2 >\dimen@
    \dimen@=\wd2 %   
  \fi
  \count@=2 %
  \def\da@bars{\dabar@\dabar@}%
  \@whiledim\count@\wd4<\dimen@\do{%
    \advance\count@\@ne
    \expandafter\def\expandafter\da@bars\expandafter{%
      \da@bars
      \dabar@ 
    }%
  }%  
  \mathrel{#3}%
  \mathrel{%   
    \mathop{\da@bars}\limits
    \ifx\\#1\\%
    \else
      _{\copy0}%
    \fi
    \ifx\\#2\\%
    \else
      ^{\copy2}%
    \fi
  }%   
  \mathrel{#4}%
}
\makeatother
\newcommand\xdashmapsto[2][]{\mathrel{\mapstochar\xdashrightarrow[#1]{#2}}}

\title{Almost homogeneous varieties of Albanese codimension one}
\author{Bruno Laurent
\thanks{Mathematisches Institut, Heinrich-Heine-Universit\"at, 40204 D\"usseldorf, Germany \newline
\url{Bruno.Laurent@hhu.de}}}
\date{}

\reversemarginpar

\begin{document}

\maketitle

\begin{abstract}
We classify almost homogeneous normal varieties of Albanese codimension $1$, defined over an arbitrary field. We prove that such a variety has a unique normal equivariant completion. Over a perfect field, the group scheme of automorphisms of this completion is smooth, except in one case in characteristic $2$, and we determine its (reduced) neutral component.
\end{abstract}

\section*{Introduction}\addcontentsline{toc}{section}{Introduction}

The varieties equipped with an action of an algebraic group and having a dense orbit are said to be almost homogeneous. Their classification is a challenging problem, for which one usually restricts to some subclasses of varieties. This can be done for example by imposing conditions on the acting groups: if we only take algebraic tori, then we get the so-called toric varieties, which have been intensively studied for the past fifty years. Instead of giving restrictions of the groups, we can also require some geometric conditions on the varieties.

For a smooth complex projective variety $X$, by Hodge symmetry the birational invariants $\dim H^1(X,\mathcal{O}_X)$ and $\dim H^0(X,\Omega_X^1)$ are equal, and their common value is called the irregularity (for a curve, this is the genus). It is also the dimension of the Albanese variety $\Alb{X}$, which is an abelian variety with a universal morphism $X \to \Alb{X}$. Over an arbitrary field, these three integers make sense but they need not be equal. Nonetheless the dimension of the Albanese variety ``behaves well'' with respect to group actions, this is why we introduce the new invariant $\dim X - \dim \Alb{X}$, which we call the Albanese codimension.

If $X$ is a smooth projective curve of genus $g$ then $\Alb{X}$ is the Jacobian of $X$, so the Albanese codimension is $1-g$, which can be arbitrarily negatively large. However, if $X$ is an almost homogeneous variety then this is a non-negative integer. The case of Albanese codimension $0$ is easy: then the almost homogeneous varieties are homogeneous and  are exactly the torsors under an abelian variety (see Proposition \ref{prop_codim0}). In this article, we focus on normal almost homogeneous varieties of Albanese codimension $1$, defined over an arbitrary field.

\bigskip
On the one hand, a central idea when studying algebraic varieties $X$ endowed with an action of an algebraic group $G$ is to find equivariant morphisms to ``simpler'' varieties. One of the main tools is the structure of equivariant morphisms to a homogeneous space: an equivariant morphism $X \to G/H$ makes $X$ be an associated fiber bundle (see \cite[Sect. 2.5]{BRI_anv} or Proposition \ref{prop_assoc_fiber_bundle}). If $X$ is almost homogeneous then the Albanese map $X \to \Alb{X}$ is an example of such a morphism, and in our situation its fibers have dimension $1$.

On the other hand, we can look for an equivariant completion $\overline{X}$ of $X$. In the article \cite{AHI2}, and in the setting of smooth complex varieties and complex linear algebraic groups, Dmitry Ahiezer studied the case when the boundary $\overline{X} \setminus X$ is the union of homogeneous divisors. He proved that if such a completion $\overline{X}$ exists, then it is unique up to isomorphism. We obtain an analogous result in Proposition \ref{prop_uniqueness_completion}, for normal varieties and smooth connected algebraic groups over an arbitrary field.

Being of Albanese codimension $1$ has an interpretation in terms of completions: over an algebraically closed field, if a normal variety $X$ is almost homogeneous (and not proper) then its Albanese codimension is $1$ if and only if $X$ admits a normal equivariant completion $\overline{X}$ such that the complement of the open orbit consists of the union of homogeneous divisors which are abelian varieties (see Proposition \ref{prop_equiv_divhom_Alb}). 

Even over an arbitrary field, we can be more precise about the equivariant completions.

\begin{theointro}\label{theo_unique_compl}
Let $X$ be a normal variety of Albanese codimension $1$ and let $G$ be a smooth connected algebraic group. If $X$ is almost homogeneous under an action of $G$ then $X$ has a unique normal equivariant completion $\overline{X}$. Moreover $\overline{X}$ is projective and regular, and the boundary $\overline{X} \setminus X$ consists of zero, one or two (disjoint) $G$-homogeneous divisors.
\end{theointro}

In \cite{AHI1}, Ahiezer proved that for a complex homogeneous variety $X$ under the action of a complex linear algebraic group (with no assumption on the Albanese codimension), the boundary $\overline{X} \setminus X$ has at most two connected components. If there are exactly two components then they are homogeneous; but one of them can be an isolated point, like for the standard action of $\GL_2$ on $\PP^2$. Ahiezer also described the variety obtained by removing one of these components, and in particular the open orbit $X$ is a $\G_m$-torsor over some homogeneous projective variety $Y$. Then $X$ is the complement of the zero section of a linearized line bundle $\mathcal{L}$ over $Y$: for example, for the action of $\GL_2$ on $\PP^2$, we have $X = \A^2 \setminus \{0\}$, the $\G_m$-torsor structure is given by the standard map $\A^2 \setminus \{0\} \to \PP^1$ and $X$ is the completion of the zero section of the tautological line bundle on $\PP^1$. Moreover $\overline{X}' = \PP(\mathcal{L} \oplus \mathcal{O}_Y)$ is also a smooth equivariant completion and the boundary $\overline{X}' \setminus X$ consists of two homogeneous divisors. In \cite{AHI2}, Ahiezer also studied the case where the boundary is one homogeneous divisor. In this situation, the completion is again a homogeneous fibration. We will recover this in our setting.

\bigskip
When trying to classify almost homogeneous varieties under a group, the simplest case to look at is when the variety is the group acting on itself by translations. Over a perfect field, a theorem of Claude Chevalley states that every smooth connected algebraic group $G$ sits in a unique exact sequence $1 \to H \to G \to A \to 1$ where $H$ is a smooth connected linear group and $A$ is an abelian variety (see \cite{CON}). The quotient $G \to A$ is the Albanese morphism. Therefore, $G$ has Albanese codimension $1$ if and only if it is an extension of an abelian variety by a smooth connected linear group of dimension $1$ (that is, a form of the multiplicative group $\G_m$ or the additive group $\G_a$).

Over an imperfect field, Chevalley's theorem does not hold anymore. This leads to the notion of pseudo-abelian varieties, introduced by Burt Totaro: these are the smooth connected groups such that every smooth connected normal affine subgroup is trivial. Chevalley's theorem implies that over a perfect field, every pseudo-abelian variety is an abelian variety. Over any imperfect field, Michel Raynaud first constructed in \cite[Exp. XVII, Prop. C.5.1]{SGA3} counterexamples of large dimension, as Weil restrictions of abelian varieties. Totaro then found new classes of counterexamples, of arbitrary dimension, where the Albanese variety is an elliptic curve (see \cite[Cor. 6.5]{TOT}). In particular, there exist pseudo-abelian varieties of Albanese codimension $1$. Totaro also proved that every pseudo-abelian variety can be written as $(A \times H) / K$ where $A$ is an abelian variety, $K$ is a (commutative) finite subgroup scheme of $A$ and $H$ is an extension of a smooth connected unipotent group by $K$ (see \cite[Lemma 6.2]{TOT} for a precise statement). However, to the best of our knowledge, pseudo-abelian varieties are not fully classified, even in dimension $2$. Still, over an arbitrary field, we can give the following description (see Theorem \ref{theo_class_section}).

\begin{propintro}
Every smooth connected algebraic group of Albanese codimension $1$ is commutative. Moreover, $G$ is such a group if and only if one of the following cases holds:
\begin{enumerate}
\item The group $G$ is given by an extension $0 \to T \to G \to A \to 0$ where $T$ is a form of the multiplicative group $\G_m$ and $A$ is an abelian variety. In particular $G$ is a semi-abelian variety.
\item There exist a form $U$ of the additive group $\G_a$ and an abelian variety $A$ such that $G$ is given by an extension $0 \to U \to G \to A \to 0$.
\item The group $G$ is a pseudo-abelian variety such that we have an exact sequence $0 \to \G_{a,\overline{k}} \to G_{\overline{k}} \to B \to 0$ where $B$ is an abelian variety.
\end{enumerate}
\end{propintro}

\bigskip
Our first examples of almost homogeneous varieties of Albanese codimension $1$ are some groups acting on themselves. By Theorem \ref{theo_unique_compl}, they admit a normal equivariant completion such that the boundary consists of homogeneous divisors. The variety obtained by possibly removing some divisors of the boundary is also almost homogeneous and of Albanese codimension $1$. The following theorem states that those are, up to the existence of a $k$-rational point, all the varieties we are interested in.

\begin{theointro}\label{theo_description}
Let $X$ be a normal variety of Albanese codimension $1$ and let $G$ be a smooth connected algebraic group acting faithfully on $X$. If $X$ is almost homogeneous then there exist a smooth connected subgroup $G' \leq G$ of Albanese codimension $1$ and a $G'$-stable open subscheme $X' \subseteq X$ which is a $G'$-torsor.
\end{theointro}

\begin{remintro}
The normal $G$-equivariant completion of $X$ is the normal $G'$-equivariant completion of $X'$. Hence $X$ is obtained from the completion of a $G'$-torsor by removing some divisors of the boundary, and the acting group $G$ can be larger than $G'$. The possibilities for $G$ will be given in Theorem \ref{theo_class_section}.
\end{remintro}

It follows from the classification that, over an algebraically closed field, every proper normal almost homogeneous variety of Albanese codimension $1$ is the projectivization of a vector bundle of rank $2$ over the Albanese variety. In dimension $2$, we get a ruled surface over an elliptic curve. The automorphism group of a ruled surface was determined by Masaki Maruyama in \cite{MAR}, by using explicit computations with coordinates. We give an analogous result in Theorem \ref{theo_aut}, with a different approach. It turns out that, except in one case which only occurs in characteristic $2$, the automorphism group is smooth. The neutral component is then the expected one. The description of the full automorphism group is still an open problem.

\bigskip
The article is organized as follows. In Section \ref{sect_preliminaries} we recall some general results on the Albanese morphism. In particular, we prove that the Albanese codimension is invariant under algebraic field extensions (Proposition \ref{prop_Albcodim_field_ext}), and that the Albanese variety of a normal almost homogeneous variety is the Albanese variety of the dense orbit (Proposition \ref{prop_Alb_qhom}). The equivariant completions of normal almost homogeneous varieties $X$ are studied in Section \ref{sect_completion}. Section \ref{sect_class} is dedicated to the classification in the case of Albanese codimension $1$. Finally, Section \ref{sect_aut} deals with the automorphism group of the completions.

\paragraph{Notations and conventions} We fix a base field $k$ and an algebraic closure $\overline{k}$. For all $k$-schemes $X$ and field extensions $K/k$, the base change $X \times_k \Spec K$ will be denoted by $X_K$. 

All morphisms between $k$-schemes are morphisms over $k$. A variety over $k$ is a separated scheme of finite type over $\Spec k$ which is geometrically integral. A curve is a variety of dimension $1$. An algebraic group over $k$ is a group scheme of finite type over $\Spec k$. A subgroup of an algebraic group is a (closed) subgroup scheme.

For a proper variety $X$, we denote by $\Autgp{X}$ its group scheme of automorphisms. This is a group scheme locally of finite type, representing the functor which associates with every $k$-scheme $S$ the abstract group $\Aut_{S\textrm{-sch}}(X \times S)$ (see \cite[Th. 3.7]{MO}).

A variety $X$ is homogeneous under the action of a smooth connected algebraic group $G$ if and only if the abstract group $G(\overline{k})$ acts transitively on the set $X(\overline{k})$. If $X$ has a $k$-rational point $x$ then, equivalently, $X$ is isomorphic to $G / G_x$ where $G_x$ is the isotropy group of $x$

It is convenient to have a definition of homogeneity for schemes which may not be geometrically irreducible or geometrically reduced (for example, divisors on a variety). Let $Y$ be an integral scheme of finite type over $k$ and let $G$ be a smooth connected algebraic group acting on $Y$. We say that $Y$ is homogeneous if for every irreducible component $Z$ of $Y_{\overline{k}}$, the abstract group $G(\overline{k})$ acts transitively on $Z(\overline{k})$. This means that the reduced subscheme $Z_{\textrm{red}}$ is a homogeneous variety.

If $X$ is a variety equipped with an action of an algebraic group $G$ then an equivariant completion of $X$ is a proper variety $\overline{X}$ equipped with an action of $G$ on $\overline{X}$ and an equivariant open immersion $X \hookrightarrow \overline{X}$.

If $\mathcal{E}$ is a quasi-coherent sheaf on a base scheme $S$ then we can consider the Grothendieck projectivization $\pi : \PP(\mathcal{E}) = \Proj (\Sym \mathcal{E}) \to S$. With this convention, the sections $\sigma : S \to \PP(\mathcal{E})$ of $\pi$ are in bijective correspondence with the invertible sheaves $\mathcal{L}$ on $S$ equipped with a surjective morphism $\mathcal{E} \twoheadrightarrow \mathcal{L}$; the bijection is given by $\sigma \mapsto \sigma^* \mathcal{O}_{\PP(\mathcal{E})}(1)$.

\paragraph{Acknowledgments} I gratefully thank Michel Brion and Stefan Schr\"oer for very useful discussions and suggestions. This work was partially supported by the research training group \textit{GRK 2240 Algebro-geometric Methods in Algebra, Arithmetic and Topology}.

\section{Preliminaries on the Albanese variety}\label{sect_preliminaries}

The Albanese morphism of a projective complex manifold is a well-known object. In \cite{SerreAlb}, Jean-Pierre Serre proved its existence in the context of algebraic varieties over an algebraically closed field. Over an arbitrary field, we have the following.

\begin{proposition}\cite[App. A]{WIT}, \cite[Th. A.5]{ACMV}
Let $X$ be a variety. There exist an abelian variety $\Albz{X}$, a $\Albz{X}$-torsor $\Albu{X}$ and a morphism $\alb{X} : X \to \Albu{X}$ satisfying the following universal property: 
\begin{center}
\parbox{0.95\textwidth}{
for every abelian variety $B_0$, for every $B_0$-torsor $B_1$ and every morphism $b : X \to B_1$, there exists a unique morphism $b_1 : \Albu{X} \to B_1$ such that $b = b_1 \circ \alb{X}$, and there exists a unique group morphism $b_0 : \Albz{X} \to B_0$ such that $b_1$ is equivariant.
}
\end{center}
We call $\Albz{X}$ the Albanese variety of $X$ and $\alb{X}$ the Albanese morphism.

The formation of $\Albz{X}$, $\Albu{X}$ and $\alb{X}$ commutes with separable algebraic field extensions, as well as with finite products of algebraic varieties.
\end{proposition}

\begin{remark}
\begin{enumerate}[wide, labelwidth=!, labelindent=0pt]
\item If $X$ has a $k$-rational point $x$ then the $k$-rational point $\alb{X}(x)$ of $\Albu{X}$ induces an isomorphism $\Albz{X} \to \Albu{X}$. In this case, the Albanese variety is simply denoted by $\Alb{X}$, we have $\alb{X}(x) = 0$ and the universal property becomes: 
\begin{center}
\parbox{0.95\textwidth}{
for every abelian variety $B$ and every morphism $b : X \to B$ such that $b(x)=0$, there exists a unique group morphism $\psi : \Alb{X} \to B$ such that $b = \psi \circ \alb{X}$.
}
\end{center}
\item For regular varieties, the Albanese variety is a birational invariant (this follows from Weil's extension theorem, see \cite[8.4 Cor. 6]{BLR}). Proposition \ref{prop_Alb_qhom} will give an analogous result for normal almost homogeneous varieties.
\item If $G$ is a smooth connected algebraic group acting on $X$ then it follows from the universal property of $\alb{G \times X}$ that there exists a unique action of $G$ on $\Albu{X}$ such that $\alb{X}$ is $G$-equivariant. Moreover, this action comes from the group morphism $\alb{G} : G \to \Alb{G}$ and a unique action of $\Alb{G}$ on $\Albu{X}$.
\end{enumerate}
\end{remark}

\begin{example}
Let $A$ be an abelian variety, let $\mathcal{E}$ be a locally free sheaf of finite rank $r \geq 1$ on $A$ and let $\pi : X = \PP(\mathcal{E}) \to A$. Since $X$ is locally isomorphic to the trivial projective bundle and every morphism from $\PP^{r-1}$ to an abelian variety is trivial, we see that $\pi$ is the Albanese morphism of $X$.
\end{example}

The formation of the Albanese variety commutes with separable algebraic field extensions, but this it not true for arbitrary algebraic field extensions. Indeed, if $G$ is a pseudo-abelian variety but not an abelian variety then the kernel of the Albanese morphism $G \to \Alb{G}$ is not smooth, while the kernel of $G_{\overline{k}} \to \Alb{G_{\overline{k}}}$ is smooth (this follows from Proposition \ref{prop_Rosenlicht_decomp}). In general, let $X$ be a variety and let $k'/k$ be a field extension. The universal property of the Albanese variety applied to $(\alb{X})_{k'} : X_{k'} \to \Albu{X}_{k'}$ gives a group morphism $\phi : \Albz{X_{k'}} \to \Albz{X}_{k'}$. In \cite[Lemma A.15]{ACMV}, Jeffrey Achter, Sebastian Casalaina-Martin and Charles Vial give a precise description of $\phi$ when $k'/k$ is a finite and purely inseparable extension; in this case, $\phi$ is surjective and its kernel if connected. Actually, we can show that this kernel is infinitesimal.

\begin{proposition}\label{prop_Albcodim_field_ext}
Let $X$ be a variety and let $k'/k$ be an algebraic field extension. The natural morphism $\Albz{X_{k'}} \to \Albz{X}_{k'}$ is a purely inseparable isogeny. In particular, $\dim \Albz{X}$ is invariant under algebraic field extensions.
\end{proposition}

\begin{proof}
We prove that we can assume that $k'/k$ is finite and purely inseparable, in order to use the result of Achter, Casalaina-Martin and Vial. The fact that the kernel is finite then follows from a simple argument.

Since the constructions commute with separable field extensions, we first can assume that $k'/k$ is purely inseparable and that $X$ has a $k$-rational point. Indeed, we start by taking the separable closure $k''$ of $k$ in $k'$. Since the extension $k''/k$ is separable, the morphism $\Albz{X_{k''}} \to \Albz{X}_{k''}$ is an isomorphism. After extension of scalars, we get the isomorphism $\Albz{X_{k''}}_{k'} \to \Albz{X}_{k'}$. Then $\phi$ is the composed morphism $\Albz{X_{k'}} \to \Albz{X_{k''}}_{k'} \to \Albz{X}_{k'}$ (by uniqueness in the universal property), so that it suffices to show that $\Albz{X_{k'}} \to \Albz{X_{k''}}_{k'}$ is a purely inseparable isogeny. Hence we can replace $k$ with $k''$ and assume that $k'/k$ is purely inseparable. In this situation, the ring $K = k' \otimes_k k_s$ is a field, the extension $K/k'$ is separable and $K/k_s$ is purely inseparable. Let $\psi : \Albz{X_K} = \Albz{X_{k'}}_K \to \Albz{X}_K = (\Albz{X}_{k'})_K$ be the morphism deduced from $\phi$ after extension of scalars, so that $\phi$ is a purely inseparable isogeny if and only if so is $\psi$. Then $\psi$ is also the morphism given by the universal property of the Albanese variety. Thus we can replace $k$ and $k'$ with $k_s$ and $K$.

We now reduce to the case where $k'/k$ is finite. Let $k_0/k$ be a finite subextension of $k'/k$ such that $\Alb{X_{k'}}$ and $\alb{X_{k'}} : X_{k'} \to \Alb{X_{k'}}$ are defined over $k_0$: we have an abelian variety $A$ defined over $k_0$ and a morphism $f : X_{k_0} \to A$ such that $A_{k'} = \Alb{X_{k'}}$ and the morphism deduced from $f$ after extension of scalars is $\alb{X_{k'}}$. Then $f$ is the Albanese morphism of $X_{k_0}$, because by the universal property of $\alb{X_{k_0}}$ the morphism $f$ factorizes as $X_{k_0} \to \Alb{X_{k_0}} \to A$, and by construction after extension of scalars to $k'$ the morphism $\Alb{X_{k_0}} \to A$ becomes an isomorphism so it is already an isomorphism over $k_0$. Moreover, $\phi$ is deduced from the morphism $\Alb{X_{k_0}} \to \Alb{X}_{k_0}$.

Therefore, we can assume that $k'/k$ is a finite (and purely inseparable). By \cite[Lemma A.15]{ACMV}, $\phi$ is surjective with connected kernel. It remains to show that this kernel is finite. By \cite[Lemma 3.10]{BRI_commut}, there exist an abelian variety $B$ defined over $k$ and a purely inseparable isogeny $u : \Alb{X_{k'}} \to B_{k'}$. But it follows from \cite[Lemma A.15]{ACMV} again that this isogeny factorizes as $\Alb{X_{k'}} \stackrel{\phi}{\to} \Alb{X}_{k'} \to B_{k'}$, so $\ker \phi$ is a subgroup of $\ker u$.
\end{proof}

The Albanese variety of a homogeneous variety can be easily described. We first need a result on algebraic groups.

\begin{proposition}\label{prop_Rosenlicht_decomp}\cite[Th. 3.2.1, 4.3.2, 4.3.4 and 5.1.1]{BRI_structure} Let $G$ be a smooth connected algebraic group.
\begin{enumerate}
\item There exists a smallest normal subgroup $\aff{G}$ such that the quotient $G / \aff{G}$ is an abelian variety. The Albanese morphism of $G$ is the quotient morphism $ G \to G / \aff{G}$. The subgroup $\aff{G}$ is affine and connected, and its formation commutes with algebraic separable field extensions. If the base field $k$ is perfect then $\aff{G}$ is smooth.
\item There exists a smallest normal subgroup $\ant{G}$ such that the quotient $G / \ant{G}$ is affine. The subgroup $\ant{G}$ is anti-affine (that is, $\mathcal{O}(\aff{G}) = k$), smooth, connected and central in $G$, and its formation commutes with arbitrary field extensions.
\item We have the Rosenlicht decomposition $G = \ant{G} \cdot \aff{G}$.
\end{enumerate}
\end{proposition}

\begin{proposition}\label{prop_Alb_hom}\cite[Sect. 3]{BRI_nonaffine} Let $X$ be a variety which is homogeneous under the action of a smooth connected algebraic group $G$. If $X$ has a $k$-rational point $x$ then $\Alb{X} = G / (\aff{G} \cdot G_x)$ and the Albanese morphism is the quotient morphism $X \simeq G / G_x \to G / (\aff{G} \cdot G_x)$.
\end{proposition}

\begin{proposition}\label{prop_Alb_qhom} Let $X$ be a variety which is almost homogeneous under the action of a smooth connected algebraic group $G$, and let $U$ be a $G$-stable open subscheme of $X$.
\begin{enumerate}[wide, labelwidth=!, labelindent=0pt]
\item The natural morphisms $\Albu{U} \to \Albu{X}$ and $\Albz{U} \to \Albz{X}$ are surjective.
\item The Albanese codimension of $X$ is non-negative, and is greater or equal to that of $U$.
\item If moreover $X$ normal then $\Albu{U} \to \Albu{X}$ and $\Albz{U} \to \Albz{X}$ are isomorphisms.
\end{enumerate}
\end{proposition}

\begin{proof}
It suffices to treat the case where $U$ is the open orbit, so that $\alb{U}$ is surjective. Let $i : U \to X$ be the inclusion. By the universal property of $\alb{U}$, there exists a unique morphism $\Albu{i}$ such that the square
\begin{tikzpicture}[baseline=(m.center)]
\matrix(m)[matrix of math nodes,
row sep=2.5em, column sep=2.5em,
text height=1.5ex, text depth=0.25ex, minimum width=3em, anchor=base, ampersand replacement=\&]
{U \&  X \\
\Albu{U} \& \Albu{X}\\};
\path[->] (m-1-1) edge node[above] {$i$}(m-1-2);
\path[->] (m-1-1) edge node[left] {$\alb{U}$}(m-2-1);
\path[->] (m-2-1) edge node[below] {$\Albu{i}$}(m-2-2);
\path[->] (m-1-2) edge node[right] {$\alb{X}$}(m-2-2);
\end{tikzpicture}
is commutative and a unique group morphism $\Albz{i} : \Albz{U} \to \Albz{X}$ such that $\Albu{i}$ is equivariant. Since the different constructions commute with separable field extensions, we can assume that $U$ has a $k$-rational point $x$, and simply write $\Alb{U}$, $\Alb{X}$ and $\Alb{i}$.

The image of $\Alb{U} \to \Alb{X}$ is an abelian subvariety $A$ of $\Alb{X}$. Since $U$ is scheme-theoretically dense in $X$ and $\alb{U}$ is surjective, $A$ is the scheme-theoretic image of $\alb{X}$. Then the restriction $X \to A$ satisfies the universal property of the Albanese map, so $A = \Alb{X}$. Hence $\Alb{U} \to \Alb{X}$ is indeed surjective and $\dim X - \dim \Alb{X} \geq \dim U - \dim \Alb{U} \geq 0$.

The statement for a normal $X$ is given in the proof of \cite[Th. 3]{BRI_nonaffine} when the base field is algebraically closed. Using results of the article \cite{BRI_anv}, we see that it remains valid over an arbitrary field. For the seek of completeness, we recall the arguments. It suffices to show that the rational map $X \dashrightarrow \Alb{U}$ induced by $\alb{U} : U \to \Alb{U}$ is defined everywhere, because it is easily checked that it is then a morphism satisfying the universal property of $\alb{X}$. The variety $X$ is normal so by \cite[Th. 1]{BRI_anv} it is covered by $G$-stable quasi-projective open subsets (which contain $U$). Hence we can assume that $X$ is quasi-projective. Since the composed map $X \dashrightarrow \Alb{U} \xrightarrow{\Alb{i}} \Alb{X}$ is defined everywhere, by a version of Zariski's Main Theorem it suffices to prove that the morphism $\Alb{i}$ is finite.

The action of $G$ on $U$ induces a group morphism $f_0 : \Alb{G} \to \Alb{U}$ such that the square
\begin{tikzpicture}[baseline=(m.center)]
\matrix(m)[matrix of math nodes,
row sep=2.5em, column sep=2.5em,
text height=1.5ex, text depth=0.25ex, minimum width=3em, anchor=base, ampersand replacement=\&]
{G \times U \&  U \\
\Alb{U} \times \Alb{U} \& \Alb{U}\\};
\path[->] (m-1-1) edge (m-1-2);
\path[->] (m-1-1) edge node[left] {$f_0 \times \alb{U}$}(m-2-1);
\path[->] (m-2-1) edge (m-2-2);
\path[->] (m-1-2) edge node[right] {$\alb{U}$}(m-2-2);
\end{tikzpicture}
is commutative (where the bottom morphism is the group law of $\Alb{U}$). Similarly, we have a group morphism $f : \Alb{G} \to \Alb{X}$, and it agrees with the composed morphism $\Alb{G} \xrightarrow{f_0} \Alb{U} \xrightarrow{\Alb{i}} \Alb{X}$. The morphism $f_0$ is the quotient $\dfrac{G}{\aff{G}} \to \dfrac{G}{\aff{G} \cdot G_x}$ so its kernel $\dfrac{\aff{G} \cdot G_x}{\aff{G}} \simeq \dfrac{\aff{G}}{\aff{G} \cap G_x}$ is affine (as a quotient of $\aff{G}$) and proper (as a subgroup of $\Alb{G}$) hence it is finite. So $f_0$ is a finite morphism. The key point is that $f$ is finite too, and this is given by \cite[Cor. 3]{BRI_anv}. Therefore $\Alb{i}$ has finite kernel, and as expected it is a finite morphism.
\end{proof}

\begin{remark}
In particular, the Albanese torsor $\Albu{X}$ is homogeneous under an action of $G$ and the Albanese morphism $X \to \Albu{X}$ is surjective and equivariant. We will often use this in the sequel, together with Proposition \ref{prop_assoc_fiber_bundle}.
\end{remark}

\begin{proposition}\label{prop_codim0}
Let $X$ be a variety of Albanese codimension $0$ and let $G$ be a smooth connected algebraic group. Then $X$ is almost homogeneous under a faithful action of $G$ if and only if $G$ is an abelian variety and $X$ is a $G$-torsor.
\end{proposition}

\begin{proof}
Assume that $X$ is almost homogeneous. We can assume that $k$ is algebraically closed, so that the open orbit is isomorphic to $G / G_x$ for some $x \in X(k)$. The Albanese variety of $X$ is then $G / (\aff{G} \cdot G_x)$. By assumption, we have $\dim G_x = \dim \aff{G} \cdot G_x$. Taking the neutral components, we get $\dim (G_x)^\circ = \dim \aff{G} \cdot (G_x)^\circ$, so that $(G_x)_{\textrm{red}}^\circ = \aff{G} \cdot (G_x)_{\textrm{red}}^\circ$. Hence $\aff{G}$ is a subgroup of the smooth connected group $(G_x)_{\textrm{red}}^\circ$, and in particular a subgroup of $G_x$. But $\aff{G}$ is a normal subgroup of $G$ and $G$ acts faithfully on $G/G_x$, so $\aff{G}$ is trivial. Therefore $G$ is an abelian variety. As the action is faithful, the subgroup $G_x$ is trivial. The open orbit is thus a proper variety (isomorphic to $G$), so it is closed in $X$ and it is the whole $X$.
\end{proof}

\section{Equivariant completions}\label{sect_completion}

We recall an important observation on equivariant morphisms to homogeneous spaces. We will apply it to the Albanese morphism in order to study the completions of almost homogeneous varieties.

\begin{proposition}\label{prop_assoc_fiber_bundle}\cite[Sect. 2.5]{BRI_anv}
Let $X$ be a variety and let $G$ be an algebraic group acting on $X$. Let $f : X \to G/H$ be a $G$-equivariant morphism where $H$ is a subgroup of $G$. Let $C = f^{-1}(0)$ be the fiber at the base point of $G/H$. Then the square
\begin{center}
\begin{tikzpicture}[baseline=(m.center)]
\matrix(m)[matrix of math nodes,
row sep=2.5em, column sep=2.5em,
text height=1.5ex, text depth=0.25ex, minimum width=3em, anchor=base, ampersand replacement=\&]
{G \times C \&  G \\
X \& G/H\\};
\path[->] (m-1-1) edge node[above] {$\pr_{1}$}(m-1-2);
\path[->] (m-1-1) edge (m-2-1);
\path[->] (m-2-1) edge node[below] {$f$}(m-2-2);
\path[->] (m-1-2) edge (m-2-2);
\end{tikzpicture}
\end{center}
is cartesian, where the left arrow is the restriction of the action $G \times X \to X$. We say that $X$ is the associated fiber bundle and we denote it by $X = G \stackrel{H}{\times} C$.

Moreover, if $M$ is an ample $H$-linearized line bundle on $C$ then $L = G \stackrel{H}{\times} M$ exists and is an $f$-ample $G$-linearized line bundle on $X$.
\end{proposition}

\begin{proposition}
Let $X$ be a normal variety of Albanese codimension $1$, which is almost homogeneous under the action of a smooth connected algebraic group $G$. Then $X$ has a $G$-linearizable line bundle which is ample relatively to $\alb{X} : X \to \Albu{X}$. In particular $X$ is quasi-projective.
\end{proposition}

\begin{proof}
By extending scalars to some finite separable algebraic field extension $k'/k$ if necessary, we can assume that the open orbit has a $k$-rational point $x$. Indeed, we can adapt the arguments of \cite[Lemma 2.10]{BRI_anv}. Assume that $L'$ is a linearizable line on $X_{k'}$ which is ample relatively to $\alb{X_{k'}} : X_{k'} \to \Albu{X_{k'}} = \Albu{X}_{k'}$. Consider the norm $L = N_{k'/k}(L')$. By \cite[Lemma 2.2]{BRI_anv} this is a linearizable line bundle on $X$. Moreover, by \cite[II, Prop. 4.6.13.ii and Prop. 6.6.1]{EGA}, $L$ is ample relatively to $\alb{X}$.

The Albanese variety is $\Alb{X} = G/H$ where $H = \aff{G} \cdot G_x$. Let $C= \alb{X}^{-1}(0)$ be the fiber at the base point of $G/H$. The main idea of the proof is that, since $C$ has dimension $1$, it has an ample line bundle. In order to use Proposition \ref{prop_assoc_fiber_bundle}, we need an $H$-linearizable line bundle. The situation would be nice if $C$ were normal because, as $H$ is an affine algebraic group, some power of any line bundle on $C$ would be linearizable. Hence, our goal is to reduce to the case where $C$ is normal. This is done by reducing to the case where $H$ is smooth.

For $n \geq 0$, we denote by $G_n$ the kernel of the iterated Frobenius morphism $F^n : G \to G^{(p^n)}$. We have $H_n = H \cap G_n$ and, for $n$ large enough, the quotient $H/H_n$ is smooth (see \cite[Exp. VII A, Prop. 8.3]{SGA3}). By \cite[Lemma 2.5]{BRI_anv}, there exists a categorical quotient $f : X \to X/G_n$ which is finite, and $X/G_n$ is a normal variety. Moreover, by \cite[Lemma 2.8]{BRI_anv}, there exists a unique action of $G/G_n$ on $X/G_n$ such that $f$ is equivariant with respect to $G \to G/G_n$. If $L'$ is an ample line bundle on $X/G_n$ then, since $f$ is affine, the pullback $L = f^* L'$ is an ample line bundle on $X$. Moreover, the morphism $f$ is equivariant so the pullback of a $G/G_n$ linearization of $L'$ is a $G$-linearization of $L$.

The morphism $\alb{X} : X \to \dfrac{G}{H}$ is $G$-equivariant so it factorizes as a $G/G_n$-equivariant morphism $b : X/G_n \to \left(\dfrac{G}{H}\right)/G_n \simeq \dfrac{G}{H \cdot G_n} \simeq \dfrac{G/G_n}{H/H_n}$. We have $\dim \dfrac{G/G_n}{H/H_n} = \dim \dfrac{G}{H} = \dim X-1 = \dim (X/G_n) - 1$ so the fiber $b^{-1}(0)$ has dimension $1$. Therefore, we can assume that $H$ is smooth.

Then the morphism $G \to G/H$ is smooth (see \cite[Prop. 2.6.5]{BRI_structure}) so by base change the morphism $G \times C \to X$ is smooth too. Since $X$ is a normal variety, $G \times C$ is normal too, as well as $C$ (see \cite[IV4, Prop. 17.3.3]{EGA}). Finally, let $M$ be an ample line bundle on $C$. By \cite[Lemma 2.9 and Prop. 2.12]{BRI_anv}, there exists some power of $M$ which is $H$-linearizable and we can use Proposition \ref{prop_assoc_fiber_bundle}.
\end{proof}

This enables us to consider equivariant completions of $X$. Indeed, we can use the following general result, which is a direct consequence of \cite[Th. 2]{BRI_anv}.

\begin{proposition}
Every normal quasi-projective variety, equipped with an action of a smooth connected algebraic group, admits a normal equivariant completion.
\end{proposition}

\begin{proposition}\label{prop_uniqueness_completion}
Let $X$ be a normal variety and let $G$ be a smooth connected algebraic group acting on $X$. If $\overline{X}$ is a normal equivariant completion of $X$ such that the boundary $\overline{X} \setminus X$ is the union of $G$-homogeneous divisors then $\overline{X}$ is unique up to isomorphism. Moreover if $X$ is regular then so is $\overline{X}$.
\end{proposition}

The proof of Proposition \ref{prop_uniqueness_completion} follows an idea of Ahiezer, given in \cite{AHI2} in the setting of complex algebraic varieties. We need the two following well-known lemmas, for which we found no suitable reference.

\begin{lemma}
Let $Y$ and $Y'$ be varieties, let $G$ be a smooth connected algebraic group acting on $Y$ and $Y'$, and let $f : Y \dashrightarrow Y'$ be an equivariant rational map. The domain of definition of $f$ is a $G$-stable open subscheme of $Y$.
\end{lemma}

\begin{proof}
We set $U = \dom(f)$. We denote by $\alpha : G \times Y \to Y$ and $\alpha' : G \times Y' \to Y'$ the two actions, and by $\iota : G \to G$ the inverse map. We want to prove $\alpha^{-1}(U) = G \times U$. Since $f$ is equivariant, the two rational maps $\id_G \times f : G \times Y \dashrightarrow G \times Y'$ and 
\[\left\{\begin{array}[c]{ccccccc}G \times Y & \xrightarrow{(\id_G,\alpha)} & G \times Y & \xdashrightarrow{\id_G \times f} & G \times Y' & \xrightarrow{\id_G \times (\alpha' \circ (\iota, \id_Y))} & G \times Y' \\ (g,y) & \xmapsto{\hphantom{(\id_G,\alpha)}} & (g,gy) & \xdashmapsto{\hphantom{\id_G \times f}} & (g,f(gy)) & \xmapsto{\hphantom{\id_G \times (\alpha' \circ (\iota, \id_Y))}} & (g,g^{-1} f(gy)) \end{array}\right.\]
agree, so they have the same domain of definition. On the one hand, the composite map $\pr_1 \circ (\id_G \times f) : G \times Y \to G$ is defined everywhere and agrees with $\pr_1 : G \times Y \to G$, so the domain of definition of $\id_G \times f$ is the one of $\pr_2 \circ (\id_G \times f) : G \times Y \dashrightarrow Y'$. Moreover $\pr_2 \circ (\id_G \times f) = f \circ \pr_2$. Hence by \cite[IV4, Prop. 20.3.11 p242]{EGA} we have $\dom(\id_G \times f) = \dom(\pr_2 \circ f) = \pr_2^{-1}(\dom(f)) = G \times U$. On the other hand, the morphism $\id_G \times (\alpha' \circ (\iota, \id_Y))$ is an automorphism so $\dom(\id_G \times f) = \dom((\id_G \times f) \circ (\id_G,\alpha)) = (\id_G,\alpha)^{-1} (\dom(\id_G \times f)) = (\id_G,\alpha)^{-1} (G \times U) = \alpha^{-1}(U)$.
\end{proof}

\begin{lemma}
Let $Y$ be a variety and let $G$ be a smooth connected algebraic group acting on $Y$. The regular locus of $Y$ is a $G$-stable open subscheme of $Y$.
\end{lemma}

\begin{proof}
We denote by $U = \reg{Y}$ the regular locus and by $\alpha : G \times Y \to Y$ the action. Since $\alpha$ is an open morphism, the image $V$ of $G \times U$ is a open subscheme of $Y$; it contains $U$ because the neutral element acts trivially. The restriction $\beta : G \times U \to V$ is flat since so is $\alpha$. The variety $G \times U$ is regular as $G$ is smooth and $U$ is regular (see \cite[IV2, Prop. 6.8.5 p152]{EGA}), ans so $V$ is regular too (see \cite[IV2, Cor. 6.5.2 p143]{EGA}). But $U$ is the largest regular open subscheme of $Y$, so $V = U$.
\end{proof}

\begin{proof}[Proof of Proposition \ref{prop_uniqueness_completion}] Assume that we have two such normal completions $\overline{X}$ and $\overline{X}'$. It suffices to show that the rational map $f : \overline{X} \dashrightarrow \overline{X}'$ induced by $\id : X \to X$ is defined everywhere (because so is the analogous map $f' : \overline{X}' \dashrightarrow \overline{X}$, hence $f' \circ f$ and $f \circ f'$ are morphisms defined everywhere which agree with $\id$ on $X$, so $f$ is an isomorphism). By assumption, $f$ is equivariant so its domain of definition $U$ is a $G$-stable open subscheme of $\overline{X}$ containing $X$. The variety $\overline{X}$ is normal and $\overline{X}'$ is proper, so $\codim_{\overline{X}} (\overline{X} \setminus U) \geq 2$. Then for every divisor $D$ in the boundary $\overline{X} \setminus X$ we have $U \cap D \neq \emptyset$, so $D \subseteq U$ because $D$ is homogeneous. Therefore $U = \overline{X}$.

Assume that $X$ is regular. Since the regular locus $\reg{\overline{X}}$ is a $G$-stable open subscheme of $\overline{X}$ containing $X$ and $\codim_{\overline{X}} (\overline{X} \setminus \reg{\overline{X}}) \geq 2$, the same argument shows that $\overline{X}$ is regular.
\end{proof}

\begin{proposition}\label{prop_equiv_divhom_Alb}
Assume that $k$ is algebraically closed. Let $X$ be a non-proper variety and let $G$ be a smooth connected algebraic group acting faithfully on $X$ such that $X$ is homogeneous. Then the Albanese codimension of $X$ is $1$ if and only if there exists a normal equivariant completion $\overline{X}$ such that the boundary $\overline{X} \setminus X$ is the union of finitely many $G$-homogeneous divisors which are abelian varieties.
\end{proposition}

\begin{proof}
We first assume that such a $\overline{X}$ exists. In order to prove that the fibers of the Albanese morphism have dimension $1$, we will prove that the intersection of a fiber with a divisor in the boundary has codimension $1$ in the fiber, and that it consists of finitely many points.

We set $A = \Alb{X}$. If we choose $x \in X(k)$, by Propositions \ref{prop_Alb_hom} and \ref{prop_Alb_qhom} we have $\Alb{\overline{X}} = A = G / (\aff{G} \cdot G_x)$ and the Albanese morphism $\alb{X} : X \to A$ is the restriction of the $G$-equivariant morphism $\alb{\overline{X}} : \overline{X} \to A$. We consider the fiber $\overline{C} = \alb{\overline{X}}^{-1}(0)$ where $0$ is the neutral element of $A$. This is a closed subscheme of $\overline{X}$ which is stable under the action of $\aff{G} \cdot G_x$. The square
\begin{center}
\begin{tikzpicture}[baseline=(m.center)]
\matrix(m)[matrix of math nodes,
row sep=2.5em, column sep=2.5em,
text height=1.5ex, text depth=0.25ex, minimum width=3em, anchor=base, ampersand replacement=\&]
{G \times \overline{C} \&  G \\
\overline{X} \& A\\};
\path[->] (m-1-1) edge node[above] {$\pr_{1}$}(m-1-2);
\path[->] (m-1-1) edge (m-2-1);
\path[->] (m-2-1) edge node[below] {$\alb{\overline{X}}$}(m-2-2);
\path[->] (m-1-2) edge (m-2-2);
\end{tikzpicture}
\end{center}
is cartesian, where the left arrow is the restriction of the action $G \times \overline{X} \to \overline{X}$.

Let $D$ be a divisor in the boundary. Then $D \cap \overline{C}$ has codimension $1$ in $\overline{C}$: indeed the morphism $G \times \overline{C} \to \overline{X}$ is faithfully flat and $D$ has codimension $1$ in $\overline{X}$ so the scheme-theoretic preimage $G \times (D \cap \overline{C})$ has codimension $1$ in $G \times \overline{C}$ (see \cite[IV2, Cor. 6.1.4]{EGA}).

By assumption, the abstract group $(\aff{G} \cdot G_x)(k)$ acts transitively on the set $D(k) \cap \overline{C}(k)$. But $D$ is an abelian variety so $\Aut_D^\circ = D$ ($D$ is an algebraic group acting on itself by translation), so the affine group $\aff{G}$ acts trivially on $D$. Moreover $\dim \aff{G} \cdot G_x = \dim \aff{G}$: by \cite[Lemma 2.1]{BRI_nonaffine} the group $G_x$ is affine so $(G_x)_{\textrm{red}}^\circ$ is a smooth connected affine group, hence its image in the abelian variety $G / \aff{G}$ is trivial. Then the set $D(k) \cap \overline{C}(k)$ consists of finitely many points and therefore $\dim \overline{C} = 1$. %We may assume that the neutral element of the abelian variety $D$ is in $D \cap \overline{C}$. Then $D \to A$ is a group morphism (see \cite[Prop. 4.1.4]{BRI_structure}) and its scheme-theoretic kernel is $D \cap \overline{C}$. 

\medskip
Conversely, assume that the Albanese codimension of $X$ is $1$. Let $\overline{X}$ be a normal equivariant completion. First, the fiber $\overline{C}$ is irreducible. Indeed, let $F$ be the closure of $C = \alb{X}^{-1}(0)$ in $\overline{C}$. If $V = \overline{C} \setminus F$ is not empty then it is an open subscheme of $\overline{C}$. As the restriction $\alpha : G \times \overline{C} \to \overline{X}$ of the action is flat, the image $\alpha(G \times V)$ is a $G$-stable open subscheme of $\overline{X}$. Then $\alpha(G \times V)$ contains the open orbit $X$, which is impossible. So, as expected, $\overline{C}$ is supported by $F$ and is irreducible.

Hence the complement of $C$ in $\overline{C}$ consists of finitely many $k$-rational points. The boundary $\overline{X} \setminus X$ is the union of the orbits of those points under $G$. Such a point $y$ is fixed by $\aff{G}$, because its orbit under $\aff{G}$ is smooth, connected and contained in $\overline{C} \setminus C$. Because of the Rosenlicht decomposition $G = \ant{G} \cdot \aff{G}$, the orbit $D$ of $y$ under $G$ is the orbit under $\ant{G}$. The group $\ant{G}$ is commutative so $D$ carries a structure of a smooth connected algebraic group with neutral element $y$. Then the restriction $D \to A$ of $\alb{\overline{X}}$ is a group morphism (see \cite[Prop. 4.1.4]{BRI_structure}). Moreover $A$ is $G$-homogeneous so the ($G$-equivariant) morphism $D \to A$ is surjective. Furthermore, the isotropy group $G_{y}$ contains $\aff{G}$ and is contained in $\aff{G} \cdot G_x$ but those two groups have the same dimension, so $\dim D = \dim G / G_{y} = \dim A$. Hence the morphism $D \to A$ is an isogeny and $D$ is an abelian variety. 
\end{proof}

We can now prove Theorem \ref{theo_unique_compl}, which we state again for the convenience of the reader.

\begin{theorem}
Let $X$ be a normal variety of Albanese codimension $1$ and let $G$ be a smooth connected algebraic group. If $X$ is almost homogeneous under a faithful action of $G$ then $X$ has a unique normal equivariant completion $\overline{X}$. Moreover $\overline{X}$ is projective and regular, and the boundary $\overline{X} \setminus X$ consists of zero, one or two (disjoint) $G$-homogeneous divisors.
\end{theorem}

\begin{proof}
We may assume that the open orbit $X_0$ is not projective. Let $\overline{X}$ be a normal projective equivariant completion of $X$. We first assume that $k$ is algebraically closed. We saw in the proof of Proposition \ref{prop_equiv_divhom_Alb} that the boundary $\overline{X} \setminus X_0$ is the union of $G$-homogeneous divisors, which are the $G$-orbits of the points in $\overline{C} \setminus C_0$ (with the obvious notations). In particular $\overline{X}$ is unique and regular. Moreover the curve $C_0$ is homogeneous under the action of $\aff{G}$, so $\overline{C} \setminus C_0$ consists of at most two points (see \cite[Th. 1.1]{LAU}).

We now come back to an arbitrary field $k$. We set $Y = \overline{X}$ and denote by $\widetilde{Y_{\overline{k}}}$ the normalization of $Y_{\overline{k}}$. There exists a unique action of $G_{\overline{k}}$ on $\widetilde{Y_{\overline{k}}}$ such that the morphism  $\widetilde{Y_{\overline{k}}} \to Y_{\overline{k}}$ is $G_{\overline{k}}$-equivariant. Then $\widetilde{Y_{\overline{k}}}$ is the unique normal equivariant completion of $X_{\overline{k}}$, so $\widetilde{Y_{\overline{k}}} \setminus (X_0)_{\overline{k}}$ consists of one or two $G_{\overline{k}}$-homogeneous divisors. Therefore $Y_{\overline{k}} \setminus (X_0)_{\overline{k}}$ also consists of one or two $G_{\overline{k}}$-homogeneous divisors, and $Y \setminus X_0$ consists of one or two $G$-homogeneous divisors.
\end{proof}

\section{Classification results}\label{sect_class}

In this section, we classify the almost homogeneous normal varieties of Albanese codimension $1$. We first prove the Theorem \ref{theo_description} of the introduction.

\begin{theorem}\label{theo_orbit_torsor}
Let $X$ be a normal variety of Albanese codimension $1$ and let $G$ be a smooth connected algebraic group acting faithfully on $X$. Then $X$ is almost homogeneous if and only if there exist a smooth connected subgroup $G' \leq G$ of Albanese codimension $1$ and a $G'$-stable open subscheme $X' \subseteq X$ which is a $G'$-torsor.
\end{theorem}

\begin{proof}
It suffices to prove that there exists a commutative smooth connected subgroup $G' \leq G$ under which $X$ is almost homogeneous. Indeed, in this case the open $G'$-orbit must be a $G'$-torsor (as the action is faithful) and hence the Albanese codimension of $G'$ is $1$. We can assume that $G$ is not commutative. Equivalently, the derived subgroup $D(G)$ (which is smooth and connected) is not trivial.

Let us first prove that $X$ is almost homogeneous under $\ant{G} \cdot D(G)$. It suffices to prove that $X_{\overline{k}}$ is almost homogeneous under $(\ant{G} \cdot D(G))_{\overline{k}}$ and, since the formation of $\ant{G}$ and $D(G)$ commutes with field extensions, we can assume that $k$ is algebraically closed. Let $X_0$ be the open $G$-orbit in $X$, let $x \in X_0(k)$ and let $A = \Alb{X} = \Alb{X_0}$. The fiber $C = \alb{X_0}^{-1}(0)$ is isomorphic to $(\aff{G} \cdot G_x) / G_x \simeq \aff{G} / (\aff{G} \cap G_x)$ so it is a smooth curve. We have the Rosenlicht decomposition $G = \ant{G} \cdot \aff{G}$. Since $\Autgp{A}^\circ = A$ has no non-trivial smooth connected affine subgroup, the action of $\aff{G}$ on $A$ is trivial. Therefore $A$ is homogeneous under the action of $\ant{G}$ and $C$ is stable under the action of $\aff{G}$ (and in particular under the action of $D(G)$). If $D(G)$ fixes a point $y \in C(k)$ then it is contained in the isotropy group $G_y$, which is impossible since $D(G)$ is normal in $G$ and $G$ acts faithfully on $X_0 \simeq G / G_y$. So $D(G)$ has no fixed point in $C$, hence $C$ is homogeneous under the action of $D(G)$. Therefore $X_0$ is homogeneous under the action of $\ant{G} \cdot D(G)$

If $T$ is a non-trivial torus in $D(G)$ then we claim that $X$ is almost homogeneous under the commutative smooth connected subgroup $\ant{G} \cdot T$. Indeed, we can again assume that $k$ is algebraically closed. The action of $T$ on $C$ cannot be trivial, otherwise $T$ would act trivially on $X_0$, in contradiction with the faithfulness of the action of $G$ on $X_0$. As $C$ is a reduced curve, this implies that $T$ has an open orbit $U$ in $C$. By Proposition \ref{prop_assoc_fiber_bundle} the restriction $\ant{G} \times C \to X_0$ of the action is an open morphism, so that the image of $\ant{G} \times U$ is an open $(\ant{G} \cdot T)$-orbit.

If $D(G)$ contains no non-trivial torus then it is a unipotent group. There exists an integer $i \geq 1$ such that the iterated derived subgroup $D^i(G)$ is non-trivial while $D^{i+1}(G) = 0$. Then $X$ is almost homogeneous under the commutative smooth connected subgroup $\ant{G} \cdot D^i(G)$.
\end{proof}

Now we know what are the possible varieties, but we still have to determine all the acting groups $G$. For this, is suffices to focus on homogeneous varieties (by considering the open orbit).

\begin{lemma}\label{lemma_torsor_or_proj}
Let $X$ be a variety of Albanese codimension $1$ and let $G$ be a smooth connected algebraic group acting faithfully on $X$. If $X$ is homogeneous and $G$ is minimal for that property (that is, $G$ contains no smooth connected strict subgroup under which $X$ is homogeneous) then either $G$ is commutative and $X$ is a $G$-torsor, or $X$ is projective.
\end{lemma}

\begin{proof}
We can assume that $G$ is not commutative. We saw in the proof of Theorem \ref{theo_orbit_torsor} that $X$ is then homogeneous under the action of $\ant{G} \cdot D(G)$. By minimality $G = \ant{G} \cdot D(G)$ and $D(D(G)) = D(G)$. In order to prove that $X$ is projective, we can assume that $k$ is algebraically closed. With the above notations, $C$ is a smooth curve homogeneous under $D(G)$. Moreover the action of $D(G)$ on $C$ is faithful because its kernel is normal in $D(G)$, but $G = \ant{G} \cdot D(G)$ and $\ant{G}$ is central in $G$, so this kernel is normal in $G$. Therefore $C \simeq \PP^1$ (see \cite[Th. 1.1]{LAU}). We have the cartesian square
\begin{center}
\begin{tikzpicture}[baseline=(m.center)]
\matrix(m)[matrix of math nodes,
row sep=2.5em, column sep=2.5em,
text height=1.5ex, text depth=0.25ex, minimum width=3em, anchor=base, ampersand replacement=\&]
{G \times C \&  G \\
X \& A\\};
\path[->] (m-1-1) edge node[above] {$\pr_{1}$}(m-1-2);
\path[->] (m-1-1) edge (m-2-1);
\path[->] (m-2-1) edge node[below] {$\alb{X}$}(m-2-2);
\path[->] (m-1-2) edge (m-2-2);
\end{tikzpicture}
\end{center}
Then $\alb{X}$ is proper because so is $\pr_1 : G \times C \to G$. Hence $X$ is a proper variety. Since it is quasi-projective (as we have $X \simeq G  / G_x$) the variety $X$ is even projective.
\end{proof}

A theorem of Armand Borel and Reinhold Remmert in the setting of K\"ahler manifolds, and of Carlos Sancho de Salas in the setting of algebraic varieties, states that over an algebraically closed field a projective variety is homogeneous under its automorphism group if and only if it is isomorphic to the product of an abelian variety and a (rational) variety which is homogeneous under a semisimple group of adjoint type (see \cite[th. 5.2]{SAN}). In our situation, we have the following analogous result.

\begin{proposition}\label{prop_case_proj}
Let $X$ be a projective variety of Albanese codimension $1$ and let $G$ be a smooth connected algebraic group acting faithfully on $X$. Then $X$ is homogeneous under the action of $G$ if and only if there exist an abelian variety $A$, an $A$-torsor $A_1$ and a smooth projective conic $C$ such that $X \simeq A_1 \times C$ and $G = \Autgp{X}^\circ \simeq A \times \Autgp{C}^\circ$. We have $A_1 = \Albu{X}$ and $A = \Albz{X}$, and $\Autgp{C}$ is a form of $\PGL_2$.
\end{proposition}

\begin{proof}
We use again the above notations. Let us first assume that $k$ is algebraically closed. We will prove $G = A \times D(G)$, $A = \ant{G}$ and $D(G) = \aff{G} = \PGL_2$. The action of $D(G)$ on $C$ is faithful and $C = \PP^1$ is homogeneous, so $D(G) = \PGL_2$. Similarly $\aff{G}$ acts faithfully on $C$ so $\aff{G}$ is a subgroup of $\Aut_C = \PGL_2$ containing $D(G)$, and then $D(G) = \aff{G}$. The subgroup $\ant{G} \cap D(G)$ is central in $D(G)$, but the (scheme-theoretic) center of $\PGL_2$ is trivial. Therefore $G = \ant{G} \times D(G)$. Let us take $x \in C(k)$. As above, the group scheme $\aff{G} \cdot G_x$ acts on $C$ and this action is faithful, so $\aff{G} \cdot G_x = \aff{G}$. Hence $A = G / (\aff{G} \cdot G_x) = \ant{G}$,  $X = A \times C$ and $G = A \times \PGL_2 = \Aut_X^\circ$.

We now come back to an arbitrary field $k$. We still have $G = \ant{G} \times D(G)$ because the formation of $\ant{G}$ and $D(G)$ commutes with field extensions. The subgroup $\ant{G}$ is an abelian variety and we denote it again by $A$. Let $K/k$ be a finite separable extension such that there exists $x \in X(K)$. Then the isotropy group $G_{K,x}$ is a subgroup of $D(G)_K$ so $X_K = A_K \times (D(G)_K / G_{K,x})$, but the quotient $(D(G)_K / G_{K,x})$ is a form of $\PP^1$ with a $K$-rational point $x$, so we have $X_K = A_K \times \PP^1_K$. The formation of the Albanese morphism $\alb{X} : X \to \Albu{X}$ commutes with separable extensions so we get the first projection $\alb{X,K} : X_K \to A_K$. We also have the second projection $\pr_2 : X_K \to \PP^1_K$. The anticanonical bundle satisfies $\omega_{X_K}^{-1} = \left( \alb{X,K}^*\omega_{A_K}^{-1}\right) \otimes \left(\pr_2^* \omega_{\PP^1_K}^{-1} \right) = \pr_2^* \omega_{\PP^1_K}^{-1}$ so it is generated by global sections and defines a morphism $X_K \to \PP^2_K$, which is the composite $X_K \xrightarrow{\pr_2} \PP^1_K \xrightarrow{i_K} \PP^2_K$ where $i_K$ is the closed immersion defined by the bundle $\omega_{\PP^1_K}^{-1} = \mathcal{O}_{\PP^1_K}(2)$. Thus $\omega_X^{-1}$ is generated by global sections and defines a morphism $X \to \PP^2_k$. Denoting by $C$ its scheme-theoretic image, this morphism factorizes as $X \xrightarrow{p} C \xrightarrow{i} \PP^2_k$, and $\pr_2$ and $i_K$ are deduced from $p$ and $i$ after scalar extension. The curve $C$ is a smooth projective conic embedded in $\PP^2_k$ by $\omega_C^{-1}$. Since $(\alb{X,K},\pr_2) : X_K \to A_K \times \PP^1_K$ is an isomorphism, so is $(\alb{X},p) : X \to \Albu{X} \times C$.
\end{proof}

\begin{theorem}\label{theo_class_section}
Let $X$ be a variety of Albanese codimension $1$ and let $G$ be a smooth connected algebraic group. Then $X$ is homogeneous under a faithful action of $G$ if and only if one of the following cases holds:
\begin{enumerate}
\item The group $G$ is a semi-abelian variety (that is, it is given by an extension $0 \to T \to G \to A \to 0$ where $T$ is a form of the multiplicative group $\G_m$ and $A$ is an abelian variety) and $X$ is a $G$-torsor.
\item There exist a form $U$ of the additive group $\G_a$ and an abelian variety $A$ such that $G$ is given by an extension $0 \to U \to G \to A \to 0$ and $X$ is a $G$-torsor.
\item\label{case_pseudo_abelian} The group $G$ is a pseudo-abelian variety such that we have an exact sequence $0 \to \G_{a,\overline{k}} \to G_{\overline{k}} \to B \to 0$ where $B$ is an abelian variety, and $X$ is a $G$-torsor.
\item There exist an abelian variety $A$ and an $A$-torsor $A_1$ such that $X = A_1 \times \A^1$ and $G = A \times (\G_a \rtimes \G_m)$.
\item There exist an abelian variety $A$, an $A$-torsor $A_1$ and a form $C$ of $\PP^1$ such that $X = A_1 \times C$ and $G = \Aut_X^\circ = A \times \Aut_C$.
\end{enumerate}
\end{theorem}

\begin{proof}
The above varieties are indeed homogeneous. Conversely, let us assume that $X$ is homogeneous under a faithful action $\alpha : G \times X \to X$. If $X$ is projective then the classification is given by Proposition \ref{prop_case_proj}. We assume from now on that $X$ is not projective, so that $G$ contains a commutative smooth connected subgroup $H$ such that $X$ is a $H$-torsor.

\begin{mystep}
\item We first assume that $k$ is algebraically closed. We use the same notations as in the proof of Proposition \ref{prop_equiv_divhom_Alb}. By the same arguments as above, the smooth curve $C = \alb{X}^{-1}(0)$ is $\aff{G}$-stable, the actions of $\aff{G}$ and $\aff{H}$ are faithful, and $C$ is a $\aff{H}$-torsor. Since $\aff{H}$ is a commutative smooth connected algebraic group, we have $\aff{H} = \G_m$ or $\G_a$. Moreover we have an exact sequence $0 \to \aff{H} \to H \to A \to 0$. If $\aff{H} = \G_m$ then the largest smooth connected algebraic group acting faithfully on $C \simeq \A^1_*$ is $\G_m$ so $\aff{G} = \G_m$, hence $G = \ant{G} \cdot \aff{G}$ is commutative (because $\ant{G}$ is a central subgroup of $G$) and $G = H$. We now assume that $\aff{H} = \G_a$. Similarly, we have $\aff{G} = \G_a$ or $\G_a \rtimes \G_m$. If $\aff{G} = \G_a$ then $G = H$. If $\aff{G} = \G_a \rtimes \G_m$ then $\ant{G} \cap \aff{G}$ is a central subgroup of $\aff{G}$, so it is trivial. Hence $G = \ant{G} \times \aff{G}$, which implies $H = A \times \aff{H}$ and $G = A \times \aff{G}$.

\item We now come back to an arbitrary field $k$. If we have an exact sequence $0 \to \G_{m,\overline{k}} \to G_{\overline{k}} \to B \to 0$ where $B$ is an abelian variety then $G_{\overline{k}}$ is a semi-abelian variety, so by \cite[lemma 5.4.3]{BRI_structure}, $G$ is already a semi-abelian variety and we have an exact sequence $0 \to T \to G \to A \to 0$ where $T$ is a form of $\G_m$ and $A$ is an abelian variety such that $A_{\overline{k}} = B$.

\medskip
Assume that we have an exact sequence $0 \to \G_{a,\overline{k}} \to G_{\overline{k}} \to B \to 0$ where $B$ is an abelian variety. Let $N$ be the largest affine smooth connected normal subgroup of $G$. If $N$ is trivial then $G$ is a pseudo-abelian variety. If $N$ is not trivial then $N_{\overline{k}}$ is an affine smooth connected normal subgroup of $G_{\overline{k}}$ so for dimension reasons we have $N_{\overline{k}} = \G_{a,\overline{k}}$. The quotient $A = G/N$ is an abelian variety because it is so over $\overline{k}$.

\medskip
We finally assume $G_{\overline{k}} = B \times (\G_{a,\overline{k}} \rtimes \G_{m,\overline{k}})$ where $B$ is an abelian variety. Let $N$ be as above. The quotient $Q = G / N$ is a pseudo-abelian variety and hence sits in a unique exact sequence $0 \to A' \to Q \to U \to 0$ where $A'$ is an abelian variety and $U$ is a commutative unipotent group (see \cite[th. 2.1]{TOT}). Let $p$ be the composite morphism $G \to Q \to U$. Then $p_{\overline{k}}(A_{\overline{k}})$ is a smooth connected subgroup of $U_{\overline{k}}$ and it is a proper algebraic group (as a quotient of $A_{\overline{k}}$) so it is trivial. Moreover $p_{\overline{k}}(\G_{a,\overline{k}} \rtimes \G_{m,\overline{k}})$ is trivial because $\G_{a,\overline{k}} \rtimes \G_{m,\overline{k}}$ has no non-trivial unipotent quotient. Hence $U$ is trivial and $Q = A'$ is an abelian variety. Because of the exact sequence $1 \to N_{\overline{k}} \to G_{\overline{k}} \to A_{\overline{k}}' \to 0$ we have $N_{\overline{k}} = \G_{a,\overline{k}} \rtimes \G_{m,\overline{k}}$. The subgroup $N \cap \ant{G}$ is trivial (since it is so after extension to $\overline{k}$) so for dimension reasons we have $G = \ant{G} \cdot N = \ant{G} \times N$. The group $\ant{G}$ is an abelian variety and, by Lemma \ref{lemma_forms_GaGm}, we already have $N = \G_a \rtimes \G_m$.

Let $K/k$ be a finite separable extension such that $X$ admits a $K$-rational point $x$. The isotropy group of $x$ is of the form $\{1\} \times M$ where $M$ is a subgroup of $N_K$. Then $X_K = A_K \times (N_K / M)$ and $N_K/M$ is a form of $\A^1_K$ which is a torsor under the subgroup $\G_{a,K}$ of $N_K$ (because it is so after extension to $\overline{k}$), so we already have $N_K / M = \A^1_K$.

The formation of the Albanese morphism $\alb{X} : X \to \Albu{X}$ commutes with separable extensions and the formation of the affinization morphism $p : X \to \Spec \mathcal{O}(X)$ commutes with arbitrary extensions. The morphism deduced from $p$ after extension of scalars is the projection $\pr_2 : X_K \to \A^1_K$, so $\Spec \mathcal{O}(X) = \A^1_k$ because $K/k$ is separable. Since $(\alb{X,K},\pr_2) : X_K \to A_K \times \A^1_K$ is an isomorphism, so is $(\alb{X},p) : X \to \Albu{X} \times \A^1_k$. \qedhere
\end{mystep}
\end{proof}

\begin{remark}
\begin{enumerate}[wide, labelwidth=!, labelindent=0pt]
\item The abelian variety $A$ appearing in the group $G$ is the Albanese variety of $X$.
\item If $G$ is a semi-abelian variety given by an exact sequence $1 \to T \to G \to A \to 0$ where $T$ is a form of $\G_m$ then, by definition of $\aff{G}$, we have $\aff{G} = T$. Since $G = \ant{G} \cdot \aff{G}$ and $\ant{G}$ is central, the group $G$ is commutative. We set $\Gamma = \Gal(k_s / k)$ and denote by $\Lambda \simeq \Z$ the $\Gamma$-module of characters of $T$. The Barsotti-Weil formula gives $\Ext^1(A,T) \simeq \Hom_{gp}^\Gamma(\Lambda,\widehat{A}(k_s))$ where $\widehat{A} = \Pic^\circ(A)$ is the dual abelian variety (see \cite[prop. 5.4.10]{BRI_structure}). If $T = \G_m$ then $\Ext^1(A,\G_m) \simeq \widehat{A}(k)$. If $T$ is a non-trivial form of $\G_m$ then there exists a Galois extension $K/k$ of degree $2$ such that $T_K \simeq \G_{m,K}$; writing $\Gamma = \{\id,\sigma\}$, we have $\Ext^1(A,T) \simeq \Hom_{gp}^{\{\id,\sigma\}}(\Z,\widehat{A}(K)) \simeq \{x \in \widehat{A}(K) \mid \sigma(x) = -x \}$.
\item If we have an exact sequence $0 \to \G_a \to G \to A \to 0$ then, as above, $G$ is commutative. Such exact sequences are classified by $\Ext^1(A,\G_a) \sim H^1(A,\mathcal{O}_A)$ (see \cite[ch. III, 17.6, rem.]{OORT}). In case $A$ is an elliptic curve $E$, we have $\omega_E = \mathcal{O}_E$ so, by Serre duality, $\Ext^1(E,\G_a) \simeq k$; therefore there exists, up to isomorphism, a unique group $G$ given by such an exact sequence and not isomorphic to the direct product $E \times \G_a$ (it is the universal vector extension of $E$).
\item The case \ref{case_pseudo_abelian} of the theorem can only occur over an imperfect field. If $G$ is a pseudo-abelian variety then the Albanese map is not so well-behaved since its fibers may be non-reduced. We recall the example given in \cite[Rem. 3.4]{BRI_anv}. Let $U$ be a form of $\G_a$ over an imperfect field. There exists an extension $0 \to \alpha_p \to H \to U \to 0$ such that $H$ contains no non-trivial smooth connected subgroup. If $E$ is a supersingular elliptic curve then the kernel of the Frobenius morphism $F : E \to E^{(p)}$ is $\alpha_p$. The group $G = E \stackrel{\alpha_p}{\times} H$ is a pseudo-abelian variety. By construction, we have the exact sequence $0 \to H \to G \to E^{(p)} \to 0$. The quotient $G \to E^{(p)}$ is the Albanese morphism.
\end{enumerate}
\end{remark}

\begin{proposition}
Let $X$ be a variety of Albanese codimension $1$ and let $G$ be a smooth connected algebraic group. If $X$ is almost homogeneous under a faithful action of $G$, if $X$ has a $k$-rational point and if $G$ is not a pseudo-abelian variety (in particular if $k$ is perfect) then the group $\aff{G}$ is smooth, $X$ is the associated fiber bundle $G \stackrel{\aff{G}}{\times} C$ where $C$ is the fiber at the neutral element of the Albanese morphism of $X$, and the unique normal equivariant completion of $X$ is the associated fiber bundle $G \stackrel{\aff{G}}{\times} \overline{C}$ where $\overline{C}$ is the regular completion of the curve $C$.
\end{proposition}

\begin{proof}
Everything is clear, except the last statement. In every case we have $\aff{G} \cdot G_x = \aff{G}$, so that we have the cartesian square 
\begin{center}
\begin{tikzpicture}[baseline=(m.center)]
\matrix(m)[matrix of math nodes,
row sep=2.5em, column sep=2.5em,
text height=1.5ex, text depth=0.25ex, minimum width=3em, anchor=base, ampersand replacement=\&]
{G \times \overline{C} \&  G \\
\overline{X} \& A = G / \aff{G}\\};
\path[->] (m-1-1) edge node[above] {$\pr_{1}$}(m-1-2);
\path[->] (m-1-1) edge (m-2-1);
\path[->] (m-2-1) edge node[below] {$\alb{\overline{X}}$}(m-2-2);
\path[->] (m-1-2) edge (m-2-2);
\end{tikzpicture}
\end{center}
where $\overline{C}$ is the fiber of $\alb{\overline{X}}$ at the base point of $A$. The quotient morphism $G \to A$ is smooth because the subgroup $\aff{G}$ is smooth. So the morphism $G \times \overline{C} \to \overline{X}$ is smooth. Moreover $\overline{X}$ is normal, so $G \times \overline{C}$ is normal too, and hence $\overline{C}$ is normal (see \cite[IV2, Cor. 6.5.4 p143]{EGA}). Therefore $\overline{C}$ is the regular completion of $C$.
\end{proof}

\begin{remark}
It follows from the classification that if $G$ is neither a pseudo-abelian variety nor given by an extension $0 \to U \to G \to A \to 0$ where $U$ is a non-trivial form of $\G_a$ (these assumptions are satisfied if $k$ is perfect), then the regular completion of $C$ is $\PP^1$. Moreover, the morphism $\alb{\overline{X}} : \overline{X} \to A$ has sections. If $D$ is the divisor given by the image of such a section and $\mathcal{O}(D)$ is the associated invertible sheaf on $\overline{X}$, then $\mathcal{E} = (\alb{\overline{X}})_* \mathcal{O}(D)$ is a locally free sheaf of rank $2$ on $A$. The natural morphism $(\alb{\overline{X}})^* \mathcal{E} = (\alb{\overline{X}})^* (\alb{\overline{X}})_* \mathcal{O}(D) \to \mathcal{O}(D)$ is surjective so it corresponds to a morphism $\overline{X} \to \PP(\mathcal{E})$, which is an isomorphism (see \cite[Prop. V.2.2]{HART} for more details, where the arguments are given for surfaces but remain valid in any dimension).
\end{remark}

\begin{theorem}\label{theo_class_PE}
We assume that the field $k$ is perfect. Let $\overline{X}$ be a normal proper variety of Albanese codimension $1$ and having a $k$-rational point. Set $A = \Alb{\overline{X}}$. Then $\overline{X}$ is almost homogeneous under a faithful action of a smooth connected algebraic group $G$ if and only if, up to isomorphism, we have $\overline{X} = \PP(\mathcal{E})$ where one of the following cases holds:%$A$ be an abelian variety, let $\mathcal{E}$ be a locally free sheaf of rank $2$ on $A$ and let $\overline{X} = \PP(\mathcal{E})$ (so that the canonical morphism $\overline{X} \to A$ is the Albanese morphism). Then $\overline{X}$ is almost homogeneous under an action of a smooth connected algebraic group $G$ if and only if, up to tensoring $\mathcal{E}$ by an invertible sheaf,
\begin{enumerate}
\item We have $\mathcal{E} = \mathcal{O}_A \oplus \mathcal{O}_A$. In this case $\overline{X} = A \times \PP^1$, and $G$ is isomorphic to $A \times \G_m$ or $A \times T$ (where $T$ is the form of $\G_m$ given as the centralizer of a separable point of degree $2$ on $\PP^1$) or $A \times \G_a$ or $A \times (\G_a \rtimes \G_m)$ or $A \times \PGL_2$ with the natural action on $\overline{X}$.
\item There exists $\mathcal{L} \in \Pic^\circ(A) \setminus\{\mathcal{O}_A\}$ such that $\mathcal{E} = \mathcal{L} \oplus \mathcal{O}_A$. In this case, the group $G$ is given by a non-split exact sequence $0 \to \G_m \to G \to A \to 0$ and $\overline{X}$ is isomorphic to the (unique) normal equivariant completion of $G$.%$\overline{X}$ is the completion of the line bundle on $A$ associated with $\mathcal{L}$, the section at infinity corresponding to the quotient $\mathcal{E} \twoheadrightarrow \mathcal{O}_A$ and the zero section corresponding to the quotient $\mathcal{E} \twoheadrightarrow \mathcal{L}$. The complement of the union of those two sections is the open orbit and is isomorphic to $G$. 
\item The sheaf $\mathcal{E}$ is indecomposable and given by an exact sequence $0 \to \mathcal{O}_A \to \mathcal{E} \to \mathcal{O}_A \to 0$. In this case, the group $G$ is given by a non-split exact sequence $0 \to \G_a \to G \to A \to 0$ and $\overline{X}$ is isomorphic to the (unique) normal equivariant completion of $G$.
\end{enumerate}
\end{theorem}

\begin{proof}
Assume first that $\overline{X}$ is almost homogeneous. The possibilities for the open orbit $X$ and the group $G$ are given by Theorem \ref{theo_class_section}. The fiber $\overline{C}$ is isomorphic to $\PP^1$ and $C$ is a homogeneous curve which is open in $\overline{C}$. In particular $C$ contains at least one $k$-rational point.

Assume that $G$ is given by an exact sequence $0 \to \G_m \to G \to A \to 0$ and $X = G$. The quotient $G \to A$ is a $\G_m$-torsor so there exists an invertible sheaf $\mathcal{L}$ on $A$ such that $G$ is the complement of the zero section in the line bundle corresponding to $\mathcal{L}$. Then $\mathcal{E} = \mathcal{L} \oplus \mathcal{O}_A$ has a natural $G$-linearization so the open immersion $G \hookrightarrow \PP(\mathcal{E})$ is equivariant. By uniqueness of the normal equivariant completion, we have $\overline{X} = \PP(\mathcal{E})$. It follows from \cite[VII.16, th.6]{SerreAGCF} that the line bundle corresponding to $\mathcal{L}$ is algebraically trivial, that is to say $\mathcal{L} \in \Pic^\circ(A)$. 

Assume that $G$ is given by an exact sequence $0 \to T \to G \to A \to 0$ and $X = G$, where $T$ is a non-trivial form of $\G_m$. There exists a quadratic Galois extension $K/k$ such that $T_K \simeq \G_m$. We have $C = T$ and the complement of $C$ in $\overline{C}$ is a point with residue field $K$, so that the complement of $X$ in $\overline{X}$ is a unique irreducible divisor. After extending the scalars to $K$, we recover the situation of the former paragraph. The complement of $X_K$ in $\overline{X}_K$ is the union of the zero section and the section at infinity, and the Galois group $\Gal(K/k)$ exchanges those two sections, which forces $\mathcal{L}$ to be trivial. Hence $\overline{X}_K \simeq A_K \times \PP^1$, and we already have $\overline{X} \simeq A \times \PP^1$.

Assume that $G$ is given by an exact sequence $0 \to \G_a \to G \to A \to 0$ and $X = G$. We have a two-dimensional representation $V$ of $\G_a$ given by $\fonction{\G_a}{\GL_2}{t}{\begin{pmatrix} 1 & t \\ 0 & 1 \end{pmatrix}}$. The line spanned by the second vector of the canonical basis is the trivial representation $V_0$ of $\G_a$. The quotient $V / V_0$ is also the trivial representation, so that $V$ is an extension $0 \to V_0 \to V \to V_0 \to 0$. Then $G \stackrel{\G_a}{\times} V \to A$ is a vector bundle of rank $2$, corresponding to a $G$-linearized sheaf $\mathcal{E}$ given by an exact sequence $0 \to \mathcal{O}_A \to \mathcal{E} \to \mathcal{O}_A \to 0$. The action on $G \stackrel{\G_a}{\times} V$ yields an action on $G \stackrel{\G_a}{\times} \PP(V) = \PP(\mathcal{E})$ such that the open immersion $G = G \stackrel{\G_a}{\times} \G_a \hookrightarrow G \stackrel{\G_a}{\times} \PP(V)$ is equivariant. By uniqueness of the normal equivariant completion, we have $\overline{X} = \PP(\mathcal{E})$.

In the remaining cases we readily have $\overline{X} \simeq A \times \PP^1$.

\medskip
Conversely, let us show that is every case, the variety $\overline{X}$ is almost homogeneous. If $\mathcal{E} = \mathcal{O}_A \oplus \mathcal{O}_A$ then there is nothing to do.

If $\mathcal{E} = \mathcal{L} \oplus \mathcal{O}_A$ where $\mathcal{L} \in \Pic^\circ(A) \setminus\{\mathcal{O}_A\}$, then by \cite[VII.16, th.6]{SerreAGCF} again, the complement of the union of the two sections is an algebraic group $G$ which sits in an exact sequence $1 \to \G_m \to G \to A \to 0$ and the action of $G$ on itself extends to an action on $\overline{X}$.

Assume that the sheaf $\mathcal{E}$ is given by an exact sequence $0 \to \mathcal{O}_A \to \mathcal{E} \to \mathcal{O}_A \to 0$. Then the vector bundle $E$ over $A$ corresponding to $\mathcal{E}$ is unipotent so we can use the results of \cite[Sect. 6.4]{BSU}, which we recall. The functor of $\G_m$-equivariant automorphisms of $E$ is represented by a group scheme $\Autgp{E}^{\G_m}$ locally of finite type. The smooth connected group $\mathcal{G} = \ant{(\Autgp{E}^{\G_m})}$ acts transitively on $A$. Let $V$ be the fiber at $0$ of $E$, which is a two-dimensional vector space. The kernel $H$ of the morphism $\mathcal{G} \to A$ is unipotent (not necessarily reduced) and the action of $H$ on $V$ is a faithful linear representation. Then the vector bundle $E$ is identified with $\mathcal{G} \stackrel{H}{\times} V$. We can choose a basis of $V$ such that the image of $H$ is $\GL(V) \simeq \GL_2$ lies in $\begin{pmatrix} 1 & * \\ 0 & 1 \end{pmatrix} \simeq \G_a$. Hence there is an action of $\G_a$ on $V$ commuting with the action of $H$, so that we get a faithful action of $\G_a$ on $E$. By construction, the action of $\mathcal{G} \cdot \G_a \leq \Autgp{E}$ on $E$ induces an action on $\PP(\mathcal{E}) = \mathcal{G} \stackrel{H}{\times} \PP(V) $ with an open orbit.
\end{proof}

\section{Groups of automorphisms}\label{sect_aut}

In Theorem \ref{theo_class_PE}, we have considered almost homogeneous varieties of the form $\pi : \overline{X} = \PP(\mathcal{E}) \to A$ and we already know the largest smooth connected group of $\Autgp{\overline{X}}$. We now determine whether $\Autgp{\overline{X}}$ is smooth or not, by computing the dimension of the Lie algebra $\Lie(\Autgp{\overline{X}})$.

\begin{example}
If $\mathcal{E}$ is an indecomposable sheaf given by an exact sequence $0 \to \mathcal{O}_A \to \mathcal{E} \to \mathcal{O}_A \to 0$ then the largest smooth connected group $G$ of $\Autgp{\overline{X}}$ sits in an exact sequence $0 \to \G_a \to G \to A \to 0$. Here we construct an example where $\dim \Lie(\Autgp{\overline{X}}) = \dim A + 2$ (and not $\dim A + 1)$, so that $\Autgp{\overline{X}}$ is not smooth. We use the fact, which will be proved in Theorem \ref{theo_aut}, that $\dim \Lie(\Autgp{\overline{X}}) = \dim A + \dim H^0(A, \Sym^2 \mathcal{E})$.

We assume that the characteristic of $k$ is $p=2$. Assume moreover that there exists an abelian variety $B$ which sits in an exact sequence $0 \to H \to B \to A \to 0$ where $H$ is isomorphic to $\alpha_p$ or $\Z/p\Z$. Let $\phi : H \to \G_a$ be an injective group morphism, so that $H$ is isomorphic to the subgroup of $\G_a$ given by the equation $x^p = \lambda x$ for some $\lambda \in k$. Let $V$ be the affine plane with a basis $(e_1,e_2)$. We consider the linear action of $H$ given by $h \mapsto \begin{pmatrix} 1 & \phi(h) \\ 0 & 1 \end{pmatrix}$. Then the associated fiber bundle $E = B \stackrel{H}{\times} V$ is a unipotent vector bundle of rank $2$ over $B/H = A$, corresponding to an indecomposable sheaf given by an exact sequence $0 \to \mathcal{O}_A \to \mathcal{E} \to \mathcal{O}_A \to 0$. In this situation, $B$ is the anti-affine radical $\ant{(\Autgp{E}^{\G_m})}$ of the group of $\G_m$-equivariant automorphisms of $E$. The vector bundle corresponding to $\Sym^2 \mathcal{E}$ is $\Sym^2 E = B \stackrel{H}{\times} \Sym^2 V$, hence we have $H^0(A, \Sym^2 \mathcal{E}) \simeq (\Sym^2 V)^H$. In the basis $(e_1 e_1, e_1 e_2, e_2 e_2)$, the action of $H$ on $\Sym^2 V$ is given by $h \mapsto \begin{pmatrix} 1 & \phi(h) & \phi(h)^2 \\ 0 & 1 & 0 \\ 0 & 0 & 1 \end{pmatrix}$. Thus $(\Sym^2 V)^H = \Vect(e_1 e_1, \lambda e_1 e_2 - e_2 e_2)$ has dimension $2$.
\end{example}
%Let us notice that this is always possible if $A$ is an elliptic curve and $k$ is separably closed, because then the $p$-torsion subgroup $A[p]$ is isomorphic to $\Z/p\Z \times \mu_p$ (ordinary case) or is a non-trivial extension of $\alpha_p$ by $\alpha_p$ (supersingular case).

\begin{theorem}\label{theo_aut}
Let $A$ be an abelian variety and let $\mathcal{E}$ be a locally free sheaf of rank $2$ on $A$ such that $\overline{X} = \PP(\mathcal{E})$ is almost homogeneous under the action of a smooth connected group.
\begin{enumerate}
\item The group $\Autgp{\overline{X}}$ is smooth, except if $p=2$, the sheaf $\mathcal{E}$ is indecomposable and given by an exact sequence $0 \to \mathcal{O}_A \to \mathcal{E} \to \mathcal{O}_A \to 0$, and the kernel of $\ant{(\Autgp{E}^{\G_m})} \to A$ is isomorphic to $\alpha_p$ or $\Z/p\Z$.
\item \label{expected_group}If $\Autgp{\overline{X}}$ is smooth then $\Autgp{\overline{X}}^\circ$ is the expected group:
\begin{enumerate}
\item If $\mathcal{E} = \mathcal{O}_A \oplus \mathcal{O}_A$ then $\overline{X} = A \times \PP^1$ and $\Autgp{\overline{X}}^\circ = A \times \PGL_2$.
\item If there exists $\mathcal{L} \in \Pic^\circ(A) \setminus\{\mathcal{O}_A\}$ such that $\mathcal{E} = \mathcal{L} \oplus \mathcal{O}_A$ then we have an exact sequence $0 \to \G_m \to \Autgp{\overline{X}}^\circ \to A \to 0$.
\item If $\mathcal{E}$ is indecomposable and given by an exact sequence $0 \to \mathcal{O}_A \to \mathcal{E} \to \mathcal{O}_A \to 0$ then we have an exact sequence $0 \to \G_a \to \Autgp{\overline{X}}^\circ \to A \to 0$.
\end{enumerate}
\item If $\Autgp{\overline{X}}$ is not smooth then $\dim \Lie(\Autgp{\overline{X}}) = \dim A + 2$.
\end{enumerate}
\end{theorem}

\begin{proof}
We can assume that $\mathcal{E}$ is not isomorphic to $\mathcal{O}_A \oplus \mathcal{O}_A$ (up to tensoring by some invertible sheaf), otherwise there is nothing to do.

We want to compute $\dim \Lie(\Autgp{\overline{X}})$. As we saw in Theorem \ref{theo_class_PE}, even when $k$ is not perfect, the largest smooth connected subgroup $G$ of $\Autgp{\overline{X}}$ has dimension $\dim A + 1$ (and the kernel of the natural morphism $G \to A$ is isomorphic to $\G_m$ or $\G_a$). In particular $\dim \Lie(\Autgp{\overline{X}}) \geq \dim A + 1$, and $\Autgp{\overline{X}}$ is smooth if and only if the inequality is an equality.

We have $\dim \Lie(\Autgp{\overline{X}}) = \dim H^0(\overline{X},T_{\overline{X}})$ where $T_{\overline{X}}$ is the tangent bundle. By definition, the relative tangent bundle $T_{\overline{X}/A}$ is given by the exact sequence $0 \to T_{\overline{X}/A} \to T_{\overline{X}} \to \pi^* T_A \to 0$. This gives an exact sequence \[0 \to H^0(\overline{X},T_{\overline{X}/A}) \to H^0(\overline{X},T_{\overline{X}}) \to H^0(\overline{X},\pi^* T_A).\]

The tangent bundle of an algebraic group is trivial so $H^0(\overline{X},\pi^* T_A) = H^0(A,T_A) \simeq \Lie(A)$. The morphism $\pi$ induces a group morphism $\Autgp{\overline{X}}^\circ \to \Autgp{A}^\circ = A$ and the arrow $H^0(\overline{X},T_{\overline{X}}) \to H^0(A,T_A)$ is identified with the differential of $\Autgp{\overline{X}}^\circ \to A$ (see \cite[p. 54]{BSU}). Since the restriction $G \to A$ is surjective with smooth kernel, this differential is surjective, thus $\dim \Lie(\Autgp{\overline{X}}) = \dim A + \dim H^0(\overline{X},T_{\overline{X}/A})$. Moreover $H^0(\overline{X},T_{\overline{X}/A}) = H^0(A, \pi_*T_{\overline{X}/A})$, $T_{\overline{X}/A} \simeq \pi^*(\det \mathcal{E}^\vee) \otimes \mathcal{O}_{\overline{X}}(2)$ and $\pi_*T_{\overline{X}/A} \simeq \det \mathcal{E}^\vee \otimes \Sym^2 \mathcal{E}$ (see \cite[ex. III.8.4]{HART}).

Assume $\mathcal{E} = \mathcal{L} \oplus \mathcal{O}_A$ where $\mathcal{L} \in \Pic^\circ(A) \setminus\{\mathcal{O}_A\}$. We have $\pi_*T_{\overline{X}/A} \simeq \det \mathcal{E}^\vee \otimes \Sym^2 \mathcal{E} \simeq \mathcal{L}^\vee \otimes (\Sym^2 \mathcal{L} \oplus \mathcal{L} \oplus \mathcal{O}_A) \simeq \left(\mathcal{L}^\vee \otimes \Sym^2 \mathcal{L}\right) \oplus \mathcal{O}_A \oplus \mathcal{L}^\vee$. The canonical isomorphisms $\left(\mathcal{L}^\vee\right)^{\otimes 2} \otimes \Sym^2 \mathcal{L} \simeq \Sym^2 \left(\mathcal{L}^\vee \otimes \mathcal{L}\right) \simeq \mathcal{O}_A$ yield $\pi_* T_{\overline{X} / A} \simeq \mathcal{L} \oplus \mathcal{O}_A \oplus \mathcal{L}^\vee$. Finally, $\mathcal{L}$ and $\mathcal{L}^\vee$ are non-trivial elements of $\Pic^0(A)$ so the cohomology spaces $H^0(A,\mathcal{L})$ and $H^0(A,\mathcal{L}^\vee)$ vanish (see \cite[II.8.vii]{MUMab}). So as expected $\dim H^0(A,\pi_*T_{\overline{X}/A}) = 1$ and $\Autgp{\overline{X}}^\circ = G$.

Assume that the sheaf $\mathcal{E}$ is indecomposable and given by an exact sequence $0 \to \mathcal{O}_A \to \mathcal{E} \to \mathcal{O}_A \to 0$. We have $\pi_*T_{\overline{X}/A} \simeq \Sym^2 \mathcal{E}$. As in the proof of Theorem \ref{theo_class_PE}, we can use the results of \cite[Sect. 6.4]{BSU}. Let $E \to A$ be the vector bundle corresponding to $\mathcal{E}$ and let $V$ be the fiber of $E$ at $0$. The kernel $H$ of the surjective group morphism $\mathcal{G} = \ant{(\Autgp{E}^{\G_m})} \to A$ is unipotent and the action of $H$ on $V$ is a faithful linear representation. Then the vector bundle $E$ is identified with $\mathcal{G} \stackrel{H}{\times} V$. The subgroup $H$ is not trivial, otherwise $E$ would be the trivial vector bundle $A \times V$. Let $(e_1,e_2)$ be a basis of $V$ and let $\phi : H \to \G_a$ be an injective group morphism such that the action of $H$ on $V$ is given by $h \mapsto \begin{pmatrix} 1 & \phi(h) \\ 0 & 1 \end{pmatrix}$. It follows from \cite[Prop. 5.5.1]{BRI_structure} that $H$ cannot be isomorphic to $\G_a$; in particular $H$ is finite and $\mathcal{G}$ is an abelian variety. The vector bundle corresponding to $\Sym^2 \mathcal{E}$ is $\Sym^2 E = \mathcal{G} \stackrel{H}{\times} \Sym^2 V$, hence we have $H^0(A, \Sym^2 \mathcal{E}) \simeq (\Sym^2 V)^H$. In the basis $(e_1 e_1, e_1 e_2, e_2 e_2)$, the action of $H$ on $\Sym^2 V$ is given by $h \mapsto \begin{pmatrix} 1 & \phi(h) & \phi(h)^2 \\ 0 & 1 & 2 \phi(h) \\ 0 & 0 & 1 \end{pmatrix}$. If $p \neq 2$ then we readily have $(\Sym^2 V)^H = \Vect(e_1 e_1)$, so $\dim H^0(A,\pi_*T_{\overline{X}/A}) = 1$ and $\Autgp{\overline{X}}^\circ = G$. If $p = 2$ then the matrix becomes $\begin{pmatrix} 1 & \phi(h) & \phi(h)^2 \\ 0 & 1 & 0 \\ 0 & 0 & 1 \end{pmatrix}$. Then $\dim (\Sym^2 V)^H = 1$ or $2$, depending on the existence of $\lambda$ and $\mu$ such that $\lambda \phi(h) + \mu \phi(h)^2 = 0$ for all $h$ (and we must have $\mu \neq 0$).
\end{proof}

\begin{remark}
\begin{enumerate}[wide, labelwidth=!, labelindent=0pt]
\item Assume that the field $k$ is perfect. Then $(\Autgp{X}^\circ)_\mathrm{red}$ is the largest smooth connected subgroup of $\Autgp{X}$. Hence the point \ref{expected_group} of Theorem \ref{theo_aut} is always true if we replace $\Autgp{X}^\circ$ with $(\Autgp{X}^\circ)_\mathrm{red}$.
\item If $\mathcal{E}$ is indecomposable and given by an exact sequence $0 \to \mathcal{O}_A \to \mathcal{E} \to \mathcal{O}_A \to 0$ then $\mathcal{G} = \ant{(\Autgp{E}^{\G_m})}$ is an abelian variety isogenous to $A$ and the kernel $H$ of $\mathcal{G} \to A$ is contained in the $p$-torsion subgroup $\mathcal{G}[p]$ (as it is isomorphic to a subgroup of $\G_a$). Hence the smoothness of $\Autgp{\overline{X}}$ depends on the structure of $\mathcal{G}[p]$.

In particular, if $A$ is an elliptic curve then, over the separable closure $k_s$, the $p$-torsion subgroup is isomorphic to $\Z/p\Z \times \mu_p$ (ordinary case) or is a non-trivial extension of $\alpha_p$ by $\alpha_p$ (supersingular case). Since $\mathcal{E}$ is indecomposable, $H$ is not trivial so it is isomorphic to $\Z/p\Z$ or $\alpha_p$. Therefore, if $p=2$ we get $\dim \Lie(\Autgp{\overline{X}}) = 3$ and $\Autgp{\overline{X}}$ is not smooth, thus recovering the result of \cite[Lemma 10]{MAR} for ruled surfaces over an elliptic curve.

\end{enumerate}
\end{remark}

\appendix

\section{Appendix}

For $d \geq 1$, we denote by $\G_a \stackrel{d}{\rtimes} \G_m$ the semidirect product where the multiplicative group $\G_m$ acts on the additive group $\G_a$ with weight $d$.

\begin{lemma}\label{lemma_forms_GaGm}
If $d \geq 1$ then the group $\G_a \stackrel{d}{\rtimes} \G_m$ has no non-trivial forms.
\end{lemma}

\begin{proof}
Let $G$ be a form of $\G_a \stackrel{d}{\rtimes} \G_m$. It suffices to prove that if this form is trivialized by a finite field extension which is separable or purely inseparable then we already have $G \simeq \G_{a,k} \stackrel{d}{\rtimes} \G_{m,k}$.

\begin{mystep}
\item Assume that there exists a finite separable extension $K/k$ such that $G_K \simeq \G_{a,K} \stackrel{d}{\rtimes} \G_{m,K}$. The formation of the unipotent radical $\Ru{G}$ commutes with separable extensions (see \cite[prop. 1.1.9 p8]{CGP}) and $\G_a$ has no non-trivial forms which are trivialized by a finite separable extension, so $\Ru{G} \simeq \G_{a,k}$. Let $T$ be a maximal torus of $G$. Then $T_K$ is a maximal torus of $G_K$ (see \cite[cor. A.2.6 p.475]{CGP}), so $T$ is a form of $\G_m$. For dimension reasons, we have $G = \G_{a,k} \cdot T$. The intersection $\G_{a,k} \cap T$ is trivial because it is so after extension to $K$, so that $G$ is a semi-direct product $\G_{a,k} \rtimes T$.

The action of $T$ on $\G_{a,k}$ by conjugation gives, for every $k$-algebra $A$, a group morphism $T(A) \to \Aut_{A-gp}(\G_{a,A})$, functorially in $A$. If $k$ has characteristic zero then $\End_{A-gp}(\G_{a,A})$ is identified with $A$ where an element $\lambda \in A$ corresponds to the multiplication by $\lambda$; if $k$ has characteristic $p > 0$ then $\End_{A-gp}(\G_{a,A})$ is identified with the non-commutative ring $A[F]$ (with multiplication determined by the rule $F \lambda = \lambda^p F$) where $F : \fonction{\G_{a,A}}{\G_{a,A}}{x}{x^p}$ is the Frobenius endomorphism (see \cite[II, \textsection 3, prop. 4.4 p.196]{DEM}). Moreover the invertible elements of $A[F]$ are the invertible elements of $A$. Therefore, since $T$ acts non-trivially on $\G_{a,k}$ (otherwise $G$ would be commutative), the action gives a non-trivial morphism $T \to \G_{m,k}$. Hence $T$ has a non-trivial character over $k$, so $T = \G_{m,k}$. The weight of the action is invariant under field extensions, so we have $G \simeq \G_{a,k} \stackrel{d}{\rtimes} \G_{m,k}$.

\item Assume that there exists a finite purely inseparable extension $K/k$ such that $G_K \simeq \G_{a,K} \stackrel{d}{\rtimes} \G_{m,K}$. As above, a maximal torus $T$ of $G$ is a form of $\G_m$ , which is trivialized by $K/k$. Since the extension is purely inseparable, we have $T \simeq \G_{m,k}$. The quotient variety $G/T$ is a form of $\A^1$ and the kernel of the action of $G$ on this quotient is the subgroup $\mu_d$ of $T$.

Let us show that $G/T \simeq \A^1_k$. Every form of $\A^1$ which is trivialized by a separable extension is trivial, so we may assume that $k$ is separably closed. If for every $k$-rational point $x$ of $G/T$ the orbit morphism $T \to G/T$ were constant then, for $t \in T(k)$, the scheme $Z$ of the fixed points of $t$ (that is, the closed subscheme of $G/T$ representing the functor which associates with every $k$-scheme $S$ the set $\{x \in (G/T)(S) \mid tx=x\}$) would satisfy $Z(k) = (G/T)(k)$. But $G/T$ is reduced and hence $(G/T)(k)$ is dense in $G/T$, so we would have $Z = G/T$, in contradiction with the non-triviality of the action of $T$ on $G/T$. Thus, we can choose $x$ such that the orbit morphism is non-constant. This morphism extends to a non-constant morphism $\PP^1_k \to \widehat{G/T}$ where $\widehat{G/T}$ is the regular completion of the curve $G/T$. L\"uroth's theorem implies $\widehat{G/T} \simeq \PP^1_k$. Therefore we have $\A^1_k \setminus \{0\} = Tx \subsetneq G/T \subsetneq \PP^1_k$, so $G/T \simeq \A^1_k$.

The action of $G$ on $G/T \simeq \A^1_k$ extends to an action on the regular completion $\PP^1_k$ fixing the complement $\infty$ of $\A^1_k$. It corresponds to a morphism $f : G \to \Autgp{\PP^1_k,\infty} \simeq \G_{a,k} \rtimes \G_{m,k}$ whose kernel is equal to $\mu_d$. After extension to $K$, it becomes $f_K : \fonction{\G_{a,K} \stackrel{d}{\rtimes} \G_{m,K}}{\G_{a,K} \rtimes \G_{m,K}}{(a,t)}{(a,t^d)}$. Let $U = f^{-1}(\G_{a,k})$. Then $U_K = f_K^{-1}(\G_{a,K})$ is the subgroup $\G_{a,K}$ of $\G_{a,K} \stackrel{d}{\rtimes} \G_{m,K}$. So $U$ is a normal subgroup of $G$, the subgroup $U \cap T$ is trivial and we have $G = U \cdot T$. Moreover the kernel of the morphism $U \to \G_{a,k}$ induced by $f$ is $U \cap \mu_d$, which is trivial. So $U \simeq \G_{a,k}$ and, as above, $G \simeq \G_{a,k} \stackrel{d}{\rtimes} \G_{m,k}$.\qedhere
\end{mystep}
\end{proof}

% ************************************************************************************ 

\end{document}